\RequirePackage{fix-cm}
\documentclass[smallextended]{svjour3}       % onecolumn (second format)
\smartqed  % flush right qed marks, e.g. at end of proof
\usepackage{graphicx}

\usepackage{anyfontsize}
\journalname{Queueing Systems}

\usepackage{cite}
\usepackage{ascmac}
\usepackage{mathtools}
\usepackage{bm}
\usepackage{arydshln}
\usepackage{subfigure}
\usepackage{tikz}
\usepackage{mathdots}
\usetikzlibrary{intersections, calc, arrows, positioning, arrows.meta, shapes}
\usepackage{comment}
\usepackage{color}

\newcommand{\Poi}{\mathrm{Poi}}
\newcommand{\dd}{\mathrm{d}}
\newcommand{\E}{\mathrm{E}}

\newcommand{\ds}{\displaystyle}
\newcommand{\trunc}{\mathrm{trunc}}
\newcommand{\q}{\hspace*{1em}}

\newcommand{\rrr}[1]{#1}
\newcommand{\red}[1]{#1}
\newcommand{\blue}[1]{#1}

\spnewtheorem{corollary}[theorem]{Corollary}{\bf}{\it}
\spnewtheorem{lemma}[theorem]{Lemma}{\bf}{\it}
\spnewtheorem{remark}[theorem]{Remark}{\bf}{\it}
\spnewtheorem{definition}[theorem]{Definition}{\bf}{\it}
\spnewtheorem{assumption}[theorem]{Assumption}{\bf}{\it}

\allowdisplaybreaks

\begin{document}

\title{Time-dependent queue length distribution in queues fed by $K$
customers in a finite interval
%\thanks{Grants or other notes
%about the article that should go on the front page should be
%placed here. General acknowledgments should be placed at the end of the article.}
}
% \subtitle{Do you have a subtitle?\\ If so, write it here}

\titlerunning{Queues fed by $K$ customers in a finite interval}        
% if too long for running head

\author{Kaito Hayashi \and Yoshiaki Inoue \and Tetsuya Takine}

%\authorrunning{Short form of author list} % if too long for running head

\institute{K Hayashi \and Y. Inoue \and T. Takine \at
Department of Information and Communications Technology, 
Graduate School of Engineering, 
The University of Osaka, Suita 565-0871, Japan.\\
% Tel.: +123-45-678910\\
% Fax: +123-45-678910\\
\email{%
hayashi23@post.comm.eng.osaka-u.ac.jp 
\and yoshiaki@comm.eng.osaka-u.ac.jp 
\and takine@comm.eng.osaka-u.ac.jp}
}

\date{Received: date / Accepted: date}
% The correct dates will be entered by the editor

\maketitle

\begin{abstract}
We consider queueing models, where customers arrive according to a
continuous-time binomial process on a finite interval. In this arrival
process, a total of $K$ customers arrive in the finite time interval
$[0,T]$, where arrival times of those $K$ customers are independent
and identically distributed according to an absolutely continuous
distribution defined by its probability density function $f(t)$ on
$(0,T]$.  To analyze the time-dependent queue length distribution of
this model, we introduce an auxiliary model with non-homogeneous
Poisson arrivals and show that the time-dependent queue length
distribution in the original model is given in terms of the
time-dependent joint distribution of the numbers of arrivals and
departures in the auxiliary model. Next, we consider a numerical
procedure for computing the time-dependent queue length distribution
in Markovian models with piecewise constant $f(t)$.  A particular
feature of our computational procedure is that the truncation error
bound can be specified as an input parameter. 
Some numerical examples are also provided.


%\PACS{PACS code1 \and PACS code2 \and more} 

\end{abstract}

\section{Introduction}

The queueing theory is applicable in diverse scenarios, ranging from
waiting in line to order at a cafeteria to waiting for billing at a
hospital. In these service systems, operational hours are
predetermined, during which customers arrive and receive service. A
notable feature of such systems is the variation in customer
arrival rates over time, often exhibiting a peak during specific
hours. Occasionally, the arrival rate may exceed the facility's
maximum service capacity, leading to a significant increase in the
queue length. Understanding these dynamics is essential for optimizing
service efficiency and enhancing customer satisfaction.

Traditional queueing analysis often assumes that the system operates
over a sufficiently long time, focusing on the steady-state 
under stability conditions.  While the steady-state analysis can
provide intuitively appealing results owing to the mathematical
tractability, its relevance to service systems with finite operating
times is limited.  This limitation arises because it is challenging
for stationary queueing models to account for the time-dependent
nature of arrival rates and the resulting temporary overloading.

\smallskip
\noindent
\emph{Time-dependent queues}
Given the significance of accounting for the time-dependent nature of
systems, various approaches have been proposed in the literature. Due
to the extensive body of research on this topic, the literature
review in this paper cannot be exhaustive. Readers may 
consult a recent survey \cite{Schwarz2016} on time-dependent queues
for a more comprehensive overview.

One representative approach to time-dependent queues employs a
deterministic queueing model \cite{May1967}, where the inflow and
outflow of customers are modeled as continuous fluid. This allows the
system dynamics to be expressed through simple differential or integral
equations. The deterministic fluid model is further extended to a
stochastic framework using diffusion processes \cite{Newell1968},
which accounts for the effects of stochastic variations on queueing
behavior.
The fluid and diffusion processes also arise as applications of the
functional law of large numbers and the functional central
limit theorem to the M${}_t$/M${}_t$/1 queue under uniform
acceleration \cite{Massey1985, Mandelbaum1995}, where M${}_t$
indicates that the arrival and service processes follow
non-homogeneous Poisson processes (NHPPs).
Furthermore, various approximate methods for M${}_t$/M${}_t$/1 and
M${}_t$/M${}_t$/$c$ queues have been reported in the literature. These
include a closure approximation for the first two moments
\cite{Rothkopf1979} and the pointwise stationary approximation
\cite{Green1991-1, Green1991-2, Whitt1991}.

While many studies analyze time-dependent queues, only a few address
scenarios that the operational interval is finite. In
\cite{Nazarathy2009}, an optimal control problem for queueing
networks over a finite time horizon is considered. Similarly, in
\cite{Jouini2022}, a scheduling problem for queues with reservations
is discussed, where a finite number of customers arrive according to a
specified distribution around their reservation times.

\smallskip
\noindent
\emph{Motivation} 
The primary motivation of this work is to address a significant yet
often overlooked aspect of the time-dependent queueing theory:
\emph{how the total number $K$ of arrivals in the operating period
$[0,T]$ affects the system performance}.  Note that the total number
$K$ of arrivals is of primary concern to managers of service systems,
as it is the most straightforward metric for making critical decisions
related to resource allocation and operational efficiency.

In time-dependent queueing models, customer arrival processes are
typically formulated as non-homogeneous counting processes, most
commonly NHPPs. Consequently, the total number $K$ of arrivals is
modeled as a random variable derived from these counting processes.
In this approach, however, it is not straightforward to
understand how variations in the total number $K$ of arrivals
influence the system congestion. For instance, it can be challenging
to examine how the maximum number of waiting customers in operational
hours is impacted by a 10\% increase in the value of $K$.

This observation leads us to explore an alternative formulation of
queueing systems that operate over a finite interval $[0,T]$.
Specifically, \emph{we treat the total number $K$ of arrivals in
$[0,T]$ as a predetermined constant}. The arrival times of these $K$
customers are modeled as independent and identically distributed
(i.i.d.) random variables, drawn from an absolutely continuous
probability distribution defined by its probability density function
(pdf) $f(t)$ (see Fig. \ref{fig:model}).

\begin{figure}[!t]
\centering
\includegraphics[keepaspectratio,width=180pt]{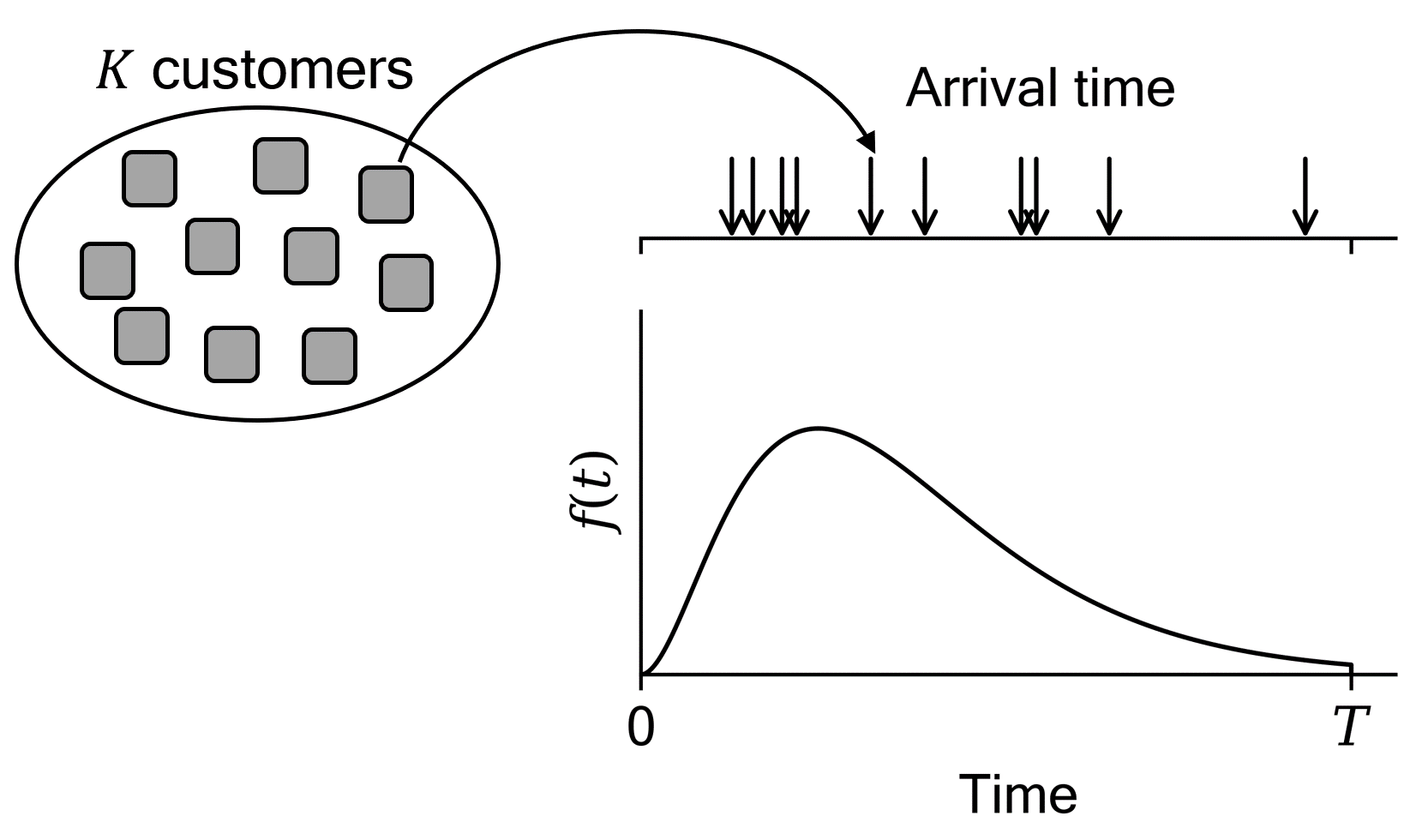}
\caption{Arrival times of $K$ customers are chosen according to 
a distribution with pdf $f(t)$ 
over the finite interval $[0,T]$.}
\label{fig:model}
\end{figure}

We refer to this arrival process as \emph{a continuous-time binomial
process} (CTBP), because the number $A(t_0, t_1)$ ($0 \leq t_0 \leq
t_1 \leq T$) of arrivals in a time-interval $(t_0, t_1]$ follows a
binomial distribution with time-dependent parameter:
\begin{align}
\Pr[A(t_0, t_1) = k] 
&= 
\binom{K}{k} 
\left(
\int_{t_0}^{t_1} f(u) \dd u
\right)^k
\left(
1 - \int_{t_0}^{t_1} f(u) \dd u
\right)^{K-k},
\nonumber
\\
&
\qquad \qquad \qquad \qquad
\qquad \qquad \qquad \qquad
k=0,1,\ldots,K.
\label{eq:CTBP-A-t0-t1}
\end{align}
The CTBP is characterized by $K$ and $f(t)$, which
enables us to treat the total number $K$ of arrivals separately from
the time-dependent likelihood $f(t)$ of arrival times.  In the CTBP,
the numbers of arrivals occurring in disjoint time intervals follow a
multinomial distribution, so that they are not independent.  This
dependence among arrivals introduces complexities in the analysis of
queueing models with CTBP arrivals.

\smallskip
\noindent
\emph{Related Work}
Queueing models with CTBP arrivals are not entirely new, but 
they have not received significant attention until recently. An
earlier study related to this topic can be found in \cite{Minh1977},
where a discrete-time model is analyzed through calculating the
transient probabilities of a Markov chain.
For models with CTBP arrivals, large-population asymptotics are 
explored in \cite{Louchard1988} and \cite{Louchard1994}, which
show weak convergence to Gaussian processes for 
infinite-server and single-server queues.

Although research on queues with CTBP arrivals stagnated after these
studies, the topic has regained interest in recent years. In
\cite{Honnappa2012, Honnappa2015}, a single-server queue with CTBP
arrivals is referred to as a $\Delta_{(i)}$/GI/$1$ queue, where the
authors derive fluid and diffusion limits for large population size
$K$ and they identify several operating regimes based on the load
level.  In \cite{Bet2019, Bet2020}, the convergence to a reflected
Brownian motion with parabolic drift is established under specific
time and spatial scaling limits. In \cite{Bet2017}, it is shown that
when service times follow a heavy-tailed distribution, the queue
length process converges to an $\alpha$-stable process as the
population size $K$ goes to infinity.  Furthermore, in
\cite{Glynn2017, Honnappa2017}, this model is termed the Random
Scatter Traffic Model (RS/G/$1$) and large deviation principles for
the workload process are established.

While these studies primarily focus on scaling limits of the
queue-length and workload processes, only a few have addressed its
exact analysis. In \cite{Bet2022}, the distribution of the number of
customers served during a single operational period is derived for
the case of exponentially distributed inter-arrival and service times, 
utilizing the representation of the system as a Markov chain. 
Additionally, the transient analysis of work in system has been
conducted in \cite{Mandjes2024}, 
deriving the double Laplace transform in time and workload.
\red{Moreover, a model in which customers may already be present in
the system at time $0$ has been considered in \cite{Boxma2025}, where
the probability generating function of the number of customers in the
system after an exponentially distributed time interval is derived.}

We also note that similar models have been considered in the context
of analyzing strategic customer behavior in queues, where $K$ 
customers probabilistically choose their arrival times based on prior
information. In \cite{Glazer1983, Haviv2015}, a game-theoretic
analysis of customer arrival times over a finite time interval is
presented, referring to the model as the ?/M/$1$ queue. 
\red{We refer the reader to \cite{Haviv2021} for a survey on strategic
customer behavior, where Example 1 explains the relationship between
the CTBP/M/1 and the M${}_t$/M/1 queues.}

Although the arrival process considered in this paper has been called
by various names, the CTBP seems to be one of the most suitable terms
for accurately describing its characteristics. Specifically, the
number of arrivals in a given interval follows a binomial distribution,
as shown in (\ref{eq:CTBP-A-t0-t1}), which closely resembles the
relationship between the Poisson process and the Poisson distribution.

\smallskip 
\noindent
\emph{Main Contributions} 
This paper establishes a general methodology for analyzing a broad
class of queueing models with CTBP arrivals. Our key result shows
that, under a fairly weak condition, the time-dependent queue length
distribution can be expressed in terms of the joint distribution of
the cumulative numbers of arrivals and departures in an auxiliary
model with ordinary NHPP arrivals, where the auxiliary model has 
the same service mechanism as the original one.  This observation
significantly simplifies the analysis of queueing models with CTBP
arrivals, reducing it to the analysis of conventional time-dependent
queueing systems with NHPP arrivals.

Moreover, we apply this methodology to a general 
\red{piecewise} Markovian queue with
the piecewise-constant pdf $f(t)$ of arrival times. For this class of
models, we derive a universal formula for an upper bound on the
numerical error caused by truncating infinite series and develop a
computational algorithm into which the truncation error bound is
incorporated. Furthermore, we present numerical examples for the
CTBP/M/$c$ queue and discuss the impact of a total number $K$ of
arrivals on the time-dependent queue length distribution.

The rest of this paper is organized as follows.  In Section
\ref{sec:model}, we introduce the general framework of the model
considered in this paper.  In Section \ref{sec:main_formula}, we
derive our main result, i.e., the connection between queues with CTBP
arrivals and the corresponding auxiliary models with NHPP arrivals.
In Section \ref{sec:application}, we demonstrate its application to
general \red{piecewise} Markovian queues with piecewise constant 
pdf of arrival times.
Specifically, we derive a universal upper bound of the truncation
error and develop a computational procedure.  We present numerical
examples of the CTBP/M/$c$ queue in Section
\ref{sec:numerics}. Finally, we conclude this paper in Section
\ref{sec:conclusion}.

\section{Model}\label{sec:model}

We consider a queueing model with CTBP arrivals. 
\blue{
We assume \rrr{that} 
the system is empty at time 0; a relaxation of this assumption 
is briefly discussed in Section \ref{sec:conclusion}.}
We also assume that a
total of $K$ ($K \in \{1, 2, \ldots\}$) customers arrive during a
finite time interval $[0,T]$ ($T>0$). Arrival times of those $K$
customers are independent and identically distributed (i.i.d.)
according to an \textit{absolutely continuous} distribution with pdf
$f(t)$, where $f(t)=0$ for all $t > T$.  We define $F(s,t)$ ($s,t \geq
0$) as
\begin{align}
F(s,t) 
= 
\begin{cases}
\ds\int_s^t f(\tau) \dd\tau, &0 \leq s \leq t,
\\[4mm]
0, & \mbox{otherwise}.
\end{cases}
\label{eq:F(s,t)}
\end{align}
By definition, we have $F(0,0)=0$ and $F(0,t)=1$ for all $t \geq T$. 
In what follows, we assume 
\[
F(0, t) > 0,
\quad
t > 0.
\]
Let $\{X_m\}_{m=1,2,\ldots,K}$ denote the ordered statistics of
arrival times, where $X_m$ denotes the arrival time of the $m$-th
customer in $(0,T]$. Since $F(0,t)$ is continuous in $t$, 
the CTBP has the orderliness property.
\[
\Pr[0 \leq X_1 < X_2 < \cdots < X_K \leq T] = 1.
\]

We define $A(t)$ ($t \geq 0$) as the number of arrivals in the
interval $[0,t]$.  By definition, we have 
\[
A(0)=0, 
\qquad
A(T)=K,
\qquad
A(s) \leq A(t),
\quad
0 \leq s \leq t.
\]
Since arrival times of $K$ customers are i.i.d.\ according
to the distribution $F(0,t)$ ($t \geq 0$), 
we have (cf.\ (\ref{eq:CTBP-A-t0-t1}))
\[
\Pr[ A(t) = k ] 
=
\binom{K}{k} \big( F(0,t) \big)^k \big( 1 - F(0,t) \big)^{K-k},
\quad
t \geq 0,\ k=0,1,\ldots,K.
\]

Let $D(t)$ ($t \geq 0$) denote the number of departures in the
interval $[0,t]$. By definition, we have 
\[
D(0)=0, 
\qquad
D(s) \leq D(t),
\quad 
0 \leq s \leq t,
\]
\red{and
\begin{equation}
D(t) \leq A(t),
\quad 
0 \leq t \leq T.
\label{eq:Dt-leq-At}
\end{equation}
\rrr{The service mechanism is asssumed to satisfy a condition that we refer
to as the \textit{conditional lack of anticipation assumption
  (conditional LAA)}.}  This assumption requires that, given the entire
sequence of customer arrival times up to time $t$, the number of
departures in the interval $[0,t]$ is conditionally independent of
arrivals occurring after time $t$.  Formally, it is stated as
follows:}

\begin{assumption}[\red{Conditional LAA}]\label{asm:service mechanism}
For any $t \geq 0$, $1 \leq k < m \leq K$, and $j = 0,1,\ldots,k$,
\begin{align}
\lefteqn{%
\Pr[D(t)=j \mid A(t)=k, (X_1,X_2,\ldots, X_m)=(x_1,x_2,\ldots,x_m)]
}\qquad
\notag \\
&=
\Pr[D(t)=j \mid A(t)=k, (X_1,X_2,\ldots, X_k)=(x_1,x_2,\ldots,x_k)].
\label{eq:asm:service mechanism}
\end{align}
\end{assumption}

\begin{remark}
In our setting, the standard LAA \cite{Wolff1982} corresponds to the 
requirement that for any $t \geq 0$, the future arrival process 
$\{A(t+u)-A(t);\, u \geq 0\}$ is independent of the past departure 
process $\{D(s);\, 0 \leq s \leq t\}$. 
In queues with CTBP arrivals, however, the total number of arrivals in 
$[0,T]$ is fixed at $K$, so that arrivals occurring after time $t$ are 
necessarily dependent on the number of arrivals in $[0,t]$. 
Therefore, instead of the usual LAA, we impose the conditional version 
of LAA as stated in Assumption~\ref{asm:service mechanism}.
\end{remark}

\begin{remark}
Assumption~\ref{asm:service mechanism} holds for standard queueing models, 
regardless of the service-time distribution or the number of servers, 
as long as future customer arrivals cannot be anticipated.  
However, it is possible to construct queueing models in which 
Assumption~\ref{asm:service mechanism} does not hold.  
For example, consider a bulk-service single-server queue in which multiple 
customers can be served simultaneously without increasing the service time.  
Suppose that the server has knowledge of future customer arrival times.  
\rrr{In such a setting, the server may choose to wait for the next
  arrival if the remaining interarrival time is shorter than a
  predetermined threshold, even when customers are already waiting and
  the server is idle.}
In this case, the conditional probability on the left-hand side of 
\eqref{eq:asm:service mechanism} depends on $X_{k+1}$, and therefore 
Assumption~\ref{asm:service mechanism} is violated.
\end{remark}

We define $L(t)$ ($t \geq 0$) as the queue length at time $t$. 
\blue{Since $L(0)=0$, we have}
\[
L(t) = A(t) - D(t) \geq 0,
\quad
t \geq 0.
\]
Our primary interest is the time-dependent probability mass function
(pmf) $\pi_{\ell}(t, K)$ of the queue length at time $t$.
\begin{align*}
\pi_{\ell}(t,K) 
=
\Pr[L(t) = \ell],
\quad 
t \geq 0,\ \ell = 0,1,\ldots,K,
\end{align*}
where we intentionally leave the model parameter $K$ in the notation
of time-dependent pmf $\pi_{\ell}(t,K)$.

\section{Connection to the model with NHPP arrivals}
\label{sec:main_formula}

In this section, we first discuss the relation between the CTBP and a
NHPP with a specific rate function. We then introduce the auxiliary
model with NHPP arrivals and show that $\pi_{\ell}(t,K)$ is given in
terms of the time-dependent joint pmf of the numbers of arrivals and
departures in the auxiliary model.

\subsection{Relation between the CTBP and NHPP}

Let $\widehat{A}(t)$ ($t \geq 0$) denote the number of arrivals in the
interval $[0, t]$ in the NHPP with rate function $\lambda(t)$. The
NHPP has the independent increment property, i.e., for $0 \leq s \leq
t$, $k=0,1,\ldots$, and $m=k,k+1,\ldots$,
\begin{align*}
\Pr[ \widehat{A}(s) = k,\ \widehat{A}(t) =m] 
&= 
\Pr[ \widehat{A}(s) = k] \cdot \Pr[\widehat{A}(t)-\widehat{A}(s) = m-k].
\end{align*}
Furthermore, 
\[
\Pr[ \widehat{A}(t) - \widehat{A}(s) = k ]
=
\Poi(\Lambda(s,t), k),
\quad
k=0,1,\ldots,\ 0 \leq s \leq t,
\]
where 
\begin{align}
\Lambda(s,t) 
&= 
\begin{cases}
\ds\int_{s}^{t}\lambda(\tau)\mathrm{d}\tau, & 0 \leq s \leq t,
\\[4mm]
0, & \mbox{otherwise},
\end{cases}
\label{eq:Lambda(s,t)}
\end{align}
and $\Poi(a,k)$ denotes the pmf of the Poisson distribution with mean
$a >0$.
\[
\Poi(a,k) = e^{-a} \frac{a^k}{k!},
\quad 
k=0,1,\ldots.
\]

In the rest of this paper, we assume that the rate function
$\lambda(t)$ is given by
\begin{equation}
\lambda(t) = \alpha f(t),
\quad
t \geq 0,\ \alpha > 0,
\label{eq:def-lambda(t)}
\end{equation}
where $\alpha$ denotes an arbitrarily chosen positive constant.
It then follows from (\ref{eq:F(s,t)}) and (\ref{eq:Lambda(s,t)}) that 
\begin{equation}
\Lambda(s,t) = \alpha F(s,t),
\quad
s,t \geq 0.
\label{eq:Lambda-F}
\end{equation}
Since $f(t)=0$ for $t >T$, no arrivals occur in the NHPP after time
$T$, as in the CTBP.

\begin{lemma}\label{lem:distribution of cumulative arrivals}
For $(t,k) \in (0,T) \times \{0,1,\ldots,K\} \cup [T,\infty) \times \{K\}$,
$m \in \{1, 2, \ldots\}$, 
$t_1, t_2, \ldots, t_m \in [0, t]$, and 
$k_1, k_2, \ldots$, $k_m \in \{0, 1, \ldots, k\}$, 
\begin{align}
\lefteqn{%
\Pr\bigl[A(t_1) = k_1, A(t_2) = k_2, \ldots, A(t_m) = k_m \mid A(t) = k\bigr]
}\qquad
\notag 
\\
&= 
\Pr\bigl[
\widehat{A}(t_1) = k_1, \widehat{A}(t_2) = k_2, \ldots, \widehat{A}(t_m) = k_m 
\mid \widehat{A}(t) = k
\bigr].
\label{eq:arrival_distribution}
\end{align}
In particular, for $t \geq T$
\begin{align}
\lefteqn{%
\Pr\bigl[A(t_1) = k_1, A(t_2) = k_2, \ldots, A(t_m) = k_m\bigr]
}\qquad
\nonumber
\\
&
=
\Pr\bigl[\widehat{A}(t_1) = k_1, 
\widehat{A}(t_2) = k_2, \ldots, \widehat{A}(t_m) = k_m 
\mid \widehat{A}(t) = K\bigr]
\nonumber
\\
&=
\Pr\bigl[\widehat{A}(t_1) = k_1, 
\widehat{A}(t_2) = k_2, \ldots, \widehat{A}(t_m) = k_m 
\mid \widehat{A}(T) = K\bigr].
\label{eq:arrival-condition:t=T}
\end{align}
\end{lemma}

The proof of Lemma \ref{lem:distribution of cumulative arrivals} is
given in Appendix \ref{appendix:lem:distribution of cumulative
arrivals}.  

\red{Let $\widehat{X}_m$ denote the arrival time of the $m$-th customer in 
the NHPP; recall that $X_m$ denotes the arrival time of the $m$-th 
customer in the CTBP.}

\begin{corollary}\label{cor:distribution of cumulative arrivals}
For $(t,k) \in (0,T) \times \{\red{1,2,\ldots,K}\} \cup [T,\infty) \times
\{K\}$ and \red{$0 < x_1 \leq x_2 \leq \cdots \leq x_k \leq t$}, we have
\begin{align}
\lefteqn{%
\Pr[X_1 \leq x_1, X_2 \leq x_2, \ldots, X_k \leq x_k \mid A(t)=k]
}\qquad
\notag
\\
&=
\Pr[
\widehat{X}_1 \leq x_1, \widehat{X}_2 \leq x_2, \ldots,\widehat{X}_k \leq x_k
\mid 
\widehat{A}(t)=k
].
\label{coro:arrival-instants}
\end{align}
In particular, for $t \geq T$,
\begin{align}
\lefteqn{%
\Pr[X_1 \leq x_1, X_2 \leq x_2, \ldots, X_K \leq x_K ]
}\qquad
\notag
\\
&=
\Pr[
\widehat{X}_1 \leq x_1, \widehat{X}_2 \leq x_2, \ldots,\widehat{X}_K \leq x_K
\mid \widehat{A}(t)=K]
\nonumber
\\
&=
\Pr[
\widehat{X}_1 \leq x_1, \widehat{X}_2 \leq x_2, \ldots,\widehat{X}_K \leq x_K
\mid \widehat{A}(T)=K].
\label{coro:arrival-instants-2}
\end{align}
\end{corollary}

\begin{proof}
It follows from Lemma \ref{lem:distribution of cumulative arrivals}
that
\begin{align*}
\lefteqn{%
\Pr[A(x_1) \geq 1, A(x_2) \geq 2, \ldots, A(x_k) \geq k \mid A(t)=k]
}\qquad
\\
&=
\Pr[\widehat{A}(x_1) \geq 1, \widehat{A}(x_2) \geq 2,
\ldots, \widehat{A}(x_k) \geq k \mid \widehat{A}(t)=k],
\end{align*}
which is equivalent to (\ref{coro:arrival-instants}).
The special case (\ref{coro:arrival-instants-2}) can be 
proved from (\ref{eq:arrival-condition:t=T}).
\qed
\end{proof}

Lemma \ref{lem:distribution of cumulative arrivals} and Corollary
\ref{cor:distribution of cumulative arrivals} imply that under the
condition that $A(t)=\widehat{A}(t)$ for some $t > 0$, the CTBP and
the NHPP with rate function $\lambda(t) = \alpha f(t)$ are
stochastically identical in the interval $[0,t]$.

\subsection{The auxiliary model and the time-dependent queue length distribution}

We now introduce the auxiliary model associated with the original
one. The arrival process in the auxiliary model is the NHPP with rate
function $\lambda(t) = \alpha f(t)$ ($\alpha > 0$).  Let
$\widehat{D}(t)$ denote the number of departures in the interval
$(0,t]$ in the auxiliary model and let $\widehat{L}(t)$ ($t \geq 0$)
denote the number of customers at time $t$. As in the original model,
we assume $\widehat{L}(0)=0$, so that
$\widehat{L}(t)=\widehat{A}(t)-\widehat{D}(t)$.  Furthermore, the
service mechanism in the auxiliary model satisfies Assumption
\ref{asm:service mechanism}, i.e., for $t \geq 0$, 
$1 \leq k < m$, and $j=0,1,\ldots,k$,
\begin{align*}
\lefteqn{%
\Pr[\widehat{D}(t)=j \mid \widehat{A}(t)=k, 
(\widehat{X}_1,\widehat{X}_2,\ldots, \widehat{X}_m)=(x_1,x_2,\ldots,x_m)]
}\qquad
\\
&=
\Pr[\widehat{D}(t)=j \mid \widehat{A}(t)=k, 
(\widehat{X}_1,\widehat{X}_2,\ldots, \widehat{X}_k)=(x_1,x_2,\ldots,x_k)].
\end{align*}

Lastly, we assume that the service mechanisms in the original
model and the auxiliary model are equivalent.  Under Assumption
\ref{asm:service mechanism}, the equivalence of the two service
mechanisms is defined formally as follows.

\begin{definition}
The service mechanisms in the original and auxiliary
models, both of which satisfy Assumption \ref{asm:service mechanism}, 
are said to be equivalent if
\begin{align}
\lefteqn{%
\Pr[D(t)=j \mid A(t)=k, (X_1,X_2,\ldots,X_k)=(x_1,x_2,\ldots,x_k)]
}\qquad
\nonumber
\\
&=
\Pr[\widehat{D}(t)=j  
\mid 
\widehat{A}(t)=k, (\widehat{X}_1,\widehat{X}_2,\ldots,\widehat{X}_k) 
= (x_1,x_2, \ldots, x_k)],
\label{eq:def-equivalence}
\end{align}
holds for all $t > 0$, $k=1,2,\ldots,K$, and $j=0,1,\ldots,k$. 
\end{definition}

In summary, the auxiliary model is such a queueing model that (i) it
has the same initial condition $\hat{L}(0)=0$ as in the original
model, (ii) customers arrive according to the NHPP with rate function
$\lambda(t)=\alpha f(t)$, (iii) it satisfies Assumption
\ref{asm:service mechanism}, and (iv) its service mechanism is
equivalent to that of the original model.

\begin{remark}
Assumption \ref{asm:service mechanism} is essential in the equivalence
of the service mechanisms. To see this, suppose $A(t)=\widehat{A}(t) =
k \leq K$ for some $t \in (0,T)$. We then have $\Pr[A(T)-A(t)=K-k]=1$
in the original model, while in the auxiliary model,
$\widehat{A}(T)-\widehat{A}(t)$ follows a Poisson distribution with
mean $\Lambda(t,T)$ independently of $\widehat{A}(t)=k$. Therefore, if
Assumption \ref{asm:service mechanism} did not hold in the original
model, it is hard to construct the auxiliary model with NHPP arrivals,
whose service mechanism is equivalent to the original one.
\end{remark}

\begin{lemma}\label{lem:service mechanisms}
For $k=0,1,\ldots,K$ and $j=0,1,\ldots,k$, we have
\begin{align}
\Pr[\widehat{D}(t)=j \mid \widehat{A}(t)=k]
&=
\Pr[D(t)=j \mid A(t)=k],
\quad
t > 0.
\label{eq:lem:service mechanisms}
\end{align}
\end{lemma}

\begin{proof}
\red{
For $k=0$, \eqref{eq:lem:service mechanisms} follows immediately from 
$\Pr[\widehat{D}(t)=0 \mid \widehat{A}(t)=0] 
= \Pr[D(t)=0 \mid A(t)=0] = 1$. 
We therefore focus on the case $k = 1,2,\ldots,K$.}

Let $\bm{x}_k = (x_1,x_2,\ldots,x_k)$. 
It then follows from (\ref{coro:arrival-instants}) 
and (\ref{eq:def-equivalence}) that
\begin{align*}
\lefteqn{%
\Pr[\widehat{D}(t)=j \mid \widehat{A}(t)=k]
}\quad
\\
&=
\int_{\bm{x}_k \in [0,t]^k}
\Pr[\widehat{D}(t)=j
\mid 
\widehat{A}(t)=k, (\widehat{X}_1, \widehat{X}_2,\ldots, \widehat{X}_k)
= \bm{x}_k]
\\
&\hspace{10em}
\dd \Pr[\widehat{X}_1 \leq x_1, 
\widehat{X}_2 \leq x_2, \ldots,\widehat{X}_k \leq x_k \mid \widehat{A}(t)=k]
\\
&=
\int_{\bm{x}_k \in [0,t]^k}
\Pr[D(t)=j
\mid 
A(t)=k, (X_1, X_2,\ldots,X_k) = \bm{x}_k]
\\
&\hspace{10em}
\dd \Pr[X_1 \leq x_1, X_2 \leq x_2, \ldots, X_k \leq x_k \mid A(t)=k]
\\
&=
\Pr[D(t)=j \mid A(t)=k].
\end{align*}
\qed 
\end{proof}

We now derive a formula for the time-dependent pmf $\pi_{\ell}(t,K)$
of the queue length in terms of the time-dependent joint pmf
$\hat{p}_{k,j}(t)$ of the numbers of arrivals and departures in the
auxiliary model.
\begin{align}
\hat{p}_{k,j}(t) 
&=
\Pr[\widehat{A}(t)=k, \widehat{D}(t)=j],
\quad
t \geq 0,\ k =0,1,\ldots,K,\ j =0,1,\ldots,k.
\label{eq:def-p_{k,j}}
\end{align}
Note here that 
\begin{equation}
\hat{p}_{k,k-\ell}(t) = \Pr[\widehat{A}(t)=k,\ \widehat{L}(t)=\ell], 
\quad 
t \geq 0,\ \ell=0,1,\ldots,\ \ k=\ell,\ell+1,\ldots,
\label{eq:p_{k,k-ell}}
\end{equation}
since $\widehat{L}(t) = \widehat{A}(t) - \widehat{D}(t)$.

\begin{theorem}\label{thm:distribution of the number of customers}
The time-dependent pmf $\pi_{\ell}(t,K)=\Pr[L(t)=\ell]$ of the queue
length in the original model is given by
\begin{align}
\pi_{\ell}(t,K)
&= 
\frac{\ds\sum_{k=\ell}^K \hat{p}_{k,k-\ell}(t) \Poi\big(\Lambda(t,T), K-k\big)}
{\ds\Poi\big(\Lambda(0,T), K\big)},
\quad
t \geq 0,\  \ell=0,1,\ldots,K,
\label{eq:main_result}
\end{align}
where $\Lambda(s,t)$ is given by \eqref{eq:Lambda(s,t)}.
\end{theorem}

\begin{remark}\label{remark:distribution}
\red{The right-hand side of \eqref{eq:main_result} can be interpreted as 
the conditional probability that $\ell$ 
customers are \rrr{in the auxiliary model}
at time $t$, given that exactly $K$ arrivals 
occur during $[0,T]$. To see this, note that \eqref{eq:p_{k,k-ell}} 
implies that the term $\hat{p}_{k,k-\ell}(t)\,\Poi\big(\Lambda(t,T), K-k\big)$ 
in the numerator of \eqref{eq:main_result} represents the joint 
probability that (i) $k$ arrivals occur during $[0,t]$, (ii) $\ell$ of 
these customers are still in the system at time $t$, 
and (iii) additional $K-k$ arrivals occur during $(t,T]$. Furthermore, because no 
arrivals occur after time $T$, \eqref{eq:main_result} reduces for 
$t > T$ to}
\[
\pi_{\ell}(t,K)
= 
\frac{\hat{p}_{K,K-\ell}(t)}{\Poi\big(\Lambda(0,T), K\big)},
\qquad
t > T,\ \ell = 0,1,\ldots,K.
\]
\end{remark}

\noindent
\textit{Proof of Theorem \ref{thm:distribution of the number of customers} }
By definition, we have
\begin{align*}
\pi_{\ell}(t,K)
&= 
\Pr[L(t)=\ell]
=
\Pr[A(t)-D(t)=\ell]
\\
&= 
\sum_{k=\ell}^K \Pr[A(t)=k, D(t)=k-\ell]
\\
&= 
\sum_{k=\ell}^K \Pr[A(t)=k] \Pr[D(t)=k-\ell \mid A(t)=k].
\end{align*}
Because (cf.\ Lemmas \ref{lem:distribution of cumulative arrivals} 
and \ref{lem:service mechanisms})
\begin{align*}
\Pr[A(t)=k] 
&= 
\Pr[\widehat{A}(t)=k \mid \widehat{A}(T)=K],
\\
\Pr[D(t)=k-\ell \mid A(t)=k] 
&= 
\Pr[\widehat{D}(t)=k-\ell \mid \widehat{A}(t)=k],
\end{align*}
we have
\begin{align*}
\pi_{\ell}(t,K)
&= 
\sum_{k=\ell}^K 
\Pr[\widehat{A}(t)=k \mid \widehat{A}(T)=K] \Pr[\widehat{D}(t)=k-\ell
\mid \widehat{A}(t)=k]
\\
&= 
\sum_{k=\ell}^K
\frac{\Pr[\widehat{A}(t)=k,\widehat{A}(T)=K]}{\Pr[\widehat{A}(T)=K]}
\cdot
\frac{\Pr[\widehat{A}(t)=k,\widehat{D}(t)=k-\ell]}{\Pr[\widehat{A}(t)=k]}
\\
&= 
\sum_{k=\ell}^K
\frac{\Pr[\widehat{A}(t)=k,\widehat{A}(T)=K]}{\Pr[\widehat{A}(t)=k]}
\cdot
\frac{\hat{p}_{k,k-\ell}(t)}{\Pr[\widehat{A}(T)=K]}
\\
&= 
\sum_{k=\ell}^K
\Pr[\widehat{A}(T)=K \mid \widehat{A}(t)=k]
\cdot 
\frac{\hat{p}_{k,k-\ell}(t)}{\Pr[\widehat{A}(T)=K]}.
\end{align*}
Since the NHPP has the independent increments property, we have
\begin{align*}
\Pr[\widehat{A}(T)=K \mid \widehat{A}(t)=k]
= 
\Pr[\widehat{A}(T)-\widehat{A}(t)=K-k].
\end{align*}
We thus obtain
\[
\pi_{\ell}(t,K)
= 
\frac{\ds\sum_{k=\ell}^K \hat{p}_{k,k-\ell}(t)
\Pr\bigl[\widehat{A}(T)-\widehat{A}(t) = K-k\bigr]}
{\Pr\bigl[\widehat{A}(T)=K\bigr]},
\]
from which the theorem follows.
\qed

\section{A computational procedure for \rrr{piecewise} Markovian queues 
with piecewise constant pdf of arrival times}
\label{sec:application}

Theorem \ref{thm:distribution of the number of customers} shows that
the computation of $\pi_{\ell}(t,K)$ is reduced to the computation of
$\hat{p}_{k,k-\ell}(t)$'s. Unfortunately, however, the latter is not
easy in general.  We thus restrict our attention to 
\red{piecewise} Markovian queues \cite{Kuczura1973},
where system parameters, including the pdf
$f(t)$ of arrival times, are assumed to be piecewise constant in time.
\red{
In the rest of this paper, we simply refer to such piecewise Markovian 
systems as Markovian queues for brevity.}
For this special case, we consider a computational procedure for the
time-dependent queue length distribution $\pi_{\ell}(t,K)$ ($0 < t
\leq T$).

\subsection{Markovian model and uniformization}

We describe a general \rrr{piecewise} Markovian model considered in
this section. We partition the interval $(0,T]$ into $N$
($N=1,2,\ldots$) disjoint intervals $(T_{n-1},T_n]$
($n=1,2,\ldots,N$), where $T_0=0$ and $T_N=T$, and we assume
\[
f(t) = \gamma_n,
\quad 
t \in (T_{n-1},T_n],\ n=1,2,\ldots,N.
\]
It then follows from (\ref{eq:def-lambda(t)}) that 
\[
\lambda(t) = \alpha \gamma_n, 
\quad 
t \in (T_{n-1},T_n],\ n=1,2,\ldots,N.
\]
Recall that $f(t)=\lambda(t) = 0$ for $t \geq T$.
For simplicity in description, let $\lambda_n = \alpha\gamma_n$
($n=1,2,\ldots,N$). It then follows from \eqref{eq:Lambda(s,t)} that 
\begin{align*}
\Lambda(t,T) 
&= \lambda_n (T_n-t) 
+ 
\sum_{m=n+1}^{N} \lambda_m (T_m-T_{m-1}),
\nonumber
\\
& \hspace*{12em}
t \in (T_{n-1},T_{n}],\ n=1,2,\ldots,N,
\end{align*}
and $\Lambda(t,T)=0$ for $t > T$ (cf. (\ref{eq:Lambda(s,t)})).

We further assume that the auxiliary model in the interval
$(T_{n-1},T_n]$ ($n=1,2,\ldots,N$) can be formulated as a
homogeneous, continuous-time absorbing Markov chain
$\{(\widehat{A}(t), \widehat{D}(t), \widehat{S}(t))\}_{T_{n-1} < t
\leq T_n}$, where 
\red{$\widehat{A}(t) \in \{0,1,\ldots\}$, 
$\widehat{D}(t) \in \{0,1,\ldots,\widehat{A}(t)\}$}, and
$\widehat{S}(t)$ takes a value in a finite set
$\widehat{\mathcal{S}}$. Note that $\widehat{S}(t)$ contains
sufficient information to describe the system behavior. 
\red{
For example, in an ordinary Markovian queue with phase-type
service times, we have to retain the service phase when the server is
busy. Furthermore, in queueing models with server vacations or
customer retrials, it is necessary to keep track of additional
auxiliary states, such as the server’s vacation status and the number
of customers in orbit.
}

\red{
Since the computation of the right-hand side of \eqref{eq:main_result} 
requires only the joint probabilities 
$\hat{p}_{k,\ell}(t)=\Pr[\widehat{A}(t)=k,\, \widehat{D}(t)=\ell]$ for 
$k=0,1,\ldots,K$, we regard all states with $\widehat{A}(t) \geq K+1$ 
as absorbing states.  
We then aggregate all absorbing states
$\{(k,j,i);\; k \in \{K+1,K+2,\ldots\},\, j \in \{0,1,\ldots,k\},\,
i \in \widehat{\mathcal{S}}\}$
into a single absorbing state.  
With this aggregation, the infinitesimal generator 
$\widehat{\bm{Q}}_n$ of 
$\{(\widehat{A}(t), \widehat{D}(t), \widehat{S}(t))\}_{T_{n-1} < t \le T_n}$ 
can be written in the following form:}
\begin{align}
\widehat{\bm{Q}}_n 
= 
\begin{bmatrix}
\bm{Q}_n & \bm{q}_n\\
\bm{0} & 0
\end{bmatrix}.
\label{eq:transition rate matrix}
\end{align}
\red{
Note that \rrr{$\bm{q}_n = (-\bm{Q}_n) \bm{e}$}, where $\bm{e}$ denotes the 
column vector of ones \rrr{with an appropriate dimension}.  
From the construction, $\bm{Q}_n$ represents the (defective) infinitesimal 
generator of the transient portion of the absorbing Markov chain 
$\{(\widehat{A}(t), \widehat{D}(t), \widehat{S}(t))\}_{T_{n-1} < t \le T_n}$ 
\rrr{and} $\bm{q}_n$ represents the transition rates to the absorbing state.
}

Let $\hat{\bm{p}}(t)$ ($t \geq 0$) denote a row vector whose $(k,j,i)$-th
element $\hat{p}_{k,j,i}(t)$ represents the time-dependent probability
of transient state $(k,j,i)$.
\begin{align*}
\hat{p}_{k,j,i}(t) 
&=
\Pr[\widehat{A}(t)=k, \widehat{D}(t)=j, \widehat{S}(t)=i],
\notag
\\
& \hspace*{8em}
k \in \{0,1,\ldots,K\},\ j \in \{0,1,\ldots,k\},\ 
i \in \widehat{\mathcal{S}}.
\end{align*}
\rrr{%
We then have
\begin{align}
\hat{\bm{p}}(t) 
&= 
\hat{\bm{p}}(T_{n-1})\exp[\bm{Q}_n (t-T_{n-1})],
\quad
t \in (T_{n-1},T_n],\ n=1,2,\ldots,N,
\label{eq:matrix-exponential-term}
\end{align}
with an initial state probability vector $\hat{\bm{p}}(0)$.
Note here that $\hat{p}_{k,j}(t)$ in (\ref{eq:def-p_{k,j}}) 
is given by
\[
\hat{p}_{k,j}(t) 
=
\sum_{i \in \widehat{\mathcal{S}}} \hat{p}_{k,j,i}(t),
\quad
k=0,1,\ldots,K,\ j=0,1,\ldots,k,
\]
which implies
\[
\sum_{k=0}^{K} 
\sum_{j=0}^{k} 
\sum_{i \in \widehat{\mathcal{S}}}
\hat{p}_{k,j,i}(t)
=
1 - \Pr[\widehat{A}(t) \geq K+1].
\]
Therefore, the computation of the pmf $\pi_{\ell}(t,K)$ in
(\ref{eq:main_result}) is reduced to the computation of 
$\hat{\bm{p}}(t)$ in (\ref{eq:matrix-exponential-term}).
}

For efficient computation of the matrix exponential terms \rrr{in
  (\ref{eq:matrix-exponential-term}),} we use the uniformization
technique \cite[pp. 154--156]{Tijms1994} as follows:
\begin{align}
\hat{\bm{p}}(t) 
&= 
\sum_{m=0}^{\infty}
\Poi(\theta_n (t-T_{n-1}),m)
\hat{\bm{p}}(T_{n-1}) \bm{P}_n^m,
\notag
\\ 
&\hspace{10em}
t \in (T_{n-1},T_n],\ n=1,2,\ldots,N,
\label{eq:uniformization}
\end{align}
where $\theta_n$ denotes the maximum of absolute values of diagonal elements 
in $\bm{Q}_n$,
\[
\theta_n = \max_{i}\left| (\bm{Q}_n)_{i,i} \right|,
\quad
n=1,2,\ldots,N,
\]
and $\bm{P}_n$ denotes a sub-stochastic matrix given by 
\[
\bm{P}_n = \bm{I} + \theta_n^{-1} \bm{Q}_n.
\]

\subsection{Truncation error bound and computational procedure}

In numerical computation, the infinite sum in
(\ref{eq:uniformization}) should be truncated. Specifically, let
\begin{equation}
\hat{\bm{p}}^{\trunc}(T_0) = \hat{\bm{p}}(0),
\label{eq:trunc-initial}
\end{equation}
and we define $\hat{\bm{p}}^{\trunc}(t)$ ($t \in (0,T]$) 
recursively as
\begin{align}
\hat{\bm{p}}^{\trunc}(t)
&= 
\sum_{m=0}^{M_n} 
\Poi(\theta_n (t-T_{n-1}),m)
\hat{\bm{p}}^{\trunc}(T_{n-1}) \bm{P}_n^m,
\notag
\\
&\hspace{10em}
t \in (T_{n-1},T_n],\ n=1,2,\ldots,N,
\label{eq:computed value of transient state probability}
\end{align}
\red{for appropriately chosen $M_n$ ($n = 1,2,\ldots,N$),
as discussed later.
}
Note here that $\hat{\bm{p}}(0) \bm{e}=1$.

Let $\hat{p}_{k,j,i}^{\trunc}(t)$ ($t \geq 0$, $k=0,1,\ldots,K$,
$j=0,1,\ldots,k$, $i \in \widehat{\mathcal{S}}$) denote the
$(k,j,i)$-th element of $\hat{\bm{p}}^{\trunc}(t)$. 
\rrr{It is clear that
\begin{equation}
\hat{p}_{k,j,i}^{\trunc}(t)
\leq 
\hat{p}_{k,j,i}(t),
\quad
t \geq 0,\ k=0,1,\ldots,K.\ j=0,1,\ldots,k,\ i \in \widehat{\mathcal{S}}.
\label{eq:trunc<original}
\end{equation}}
We define $\hat{p}_{k,j}^{\trunc}(t)$ as
\[
\hat{p}_{k,j}^{\trunc}(t)
=
\sum_{i \in \widehat{\mathcal{S}}} \hat{p}_{k,j,i}^{\trunc}(t),
\quad
t \geq 0,\ k=0,1,\ldots,K,\ j=0,1,\ldots,k.
\]
Furthermore, with $\hat{p}_{k,j}^{\trunc}(t)$, we define
$\pi_{\ell}^{\trunc}(t,K)$ as \red{(cf. \eqref{eq:main_result})}
\begin{align}
\pi_{\ell}^{\trunc}(t,K) 
&= 
\frac{\ds\sum_{k=\ell}^K \hat{p}_{k,k-\ell}^{\trunc}(t) 
\Poi\big(\Lambda(t,T),K-k\big)}
{\Poi\big(\Lambda(0,T),K\big)},
\notag
\\
&\hspace{10em} 
t \in (T_{n-1},T_n],\ n=1,2,\ldots,N,
\label{eq:computed value of distribution of the number of customers}
\end{align}
where $\Lambda(s,t)$ ($0 \leq s <t$) is given by (\ref{eq:Lambda(s,t)}).

We define $\bm{\pi}(t,K)$ and $\bm{\pi}^{\trunc}(t,K)$ 
(\red{$0 < t \leq T$}) as 
\begin{align*}
\bm{\pi}(t,K) 
&=
[\pi_0(t,K),\pi_1(t,K),\ldots,\pi_K(t,K)],
\\
\bm{\pi}^{\trunc}(t,K) 
&=
[\pi^{\trunc}_0(t,K),\pi^{\trunc}_1(t,K),\ldots,\pi^{\trunc}_K(t,K)].
\end{align*}
Our goal is to set $M_n$
($n=1,2,\ldots,N$) in such a way that for a predefined small $\epsilon
> 0$,
\[
\| \bm{\pi}(t,K) - \bm{\pi}^{\trunc}(t,K)\|_1 < \epsilon.
\]
Because $\pi_{\ell}(t,K) \geq \pi_{\ell}^{\trunc}(t,K) \geq 0$ 
($\ell=0,1,\ldots,K$), we have
\begin{align}
\| \bm{\pi}(t,K) - \bm{\pi}^{\trunc}(t,K)\|_1 
&=
\bm{\pi}(t,K)\bm{e} - \bm{\pi}^{\trunc}(t,K)\bm{e}
\nonumber
\\
&=
1 
-
\sum_{\ell=0}^K 
\frac{\ds\sum_{k=\ell}^K \hat{p}_{k,k-\ell}^{\trunc}(t) \Poi\big(\Lambda(t,T),K-k\big)}
{\Poi\big(\Lambda(0,T),K\big)}
\nonumber
\\
&=
1
-
\sum_{k=0}^K
\frac{\Poi\big(\Lambda(t,T),K-k\big)}
{\Poi\big(\Lambda(0,T),K\big)}
\sum_{\ell=0}^k \hat{p}_{k,k-\ell}^{\trunc}(t)
\nonumber
\\
&=
1
-
\sum_{k=0}^K
\frac{\Poi\big(\Lambda(t,T),K-k\big)}
{\Poi\big(\Lambda(0,T),K\big)}
\hat{p}_k^{\trunc}(t),
\quad
\red{0 < t \leq T}
\label{eq:1-norm}
\end{align}
where
\[
\hat{p}_k^{\trunc}(t)
=
\sum_{j=0}^k \hat{p}_{k,j}^{\trunc}(t),
\quad
\red{0 < t \leq T},\ k=0,1,\ldots,K.
\]
\red{
Note that in the absence of truncation,
\begin{equation}
\hat{p}_k(t) 
= 
\sum_{j=0}^{k} \hat{p}_{k,j}(t)
=
\Pr[\widehat{A}(t) = k], 
\quad
0 < t \leq T,\ k=0,1,\ldots,K.
\label{eq:hat{p}_k(t)}
\end{equation}}
\begin{lemma}\label{lem:error 1 part}
\red{
Suppose that for a given $\epsilon_0 \in (0,1]$,
the truncation points $M_n$ ($n=1,2,\ldots,N$) are chosen such that
\begin{align}
\sum_{m=0}^{M_n-K}
\Poi\bigl(\theta_n(T_n-T_{n-1}),m\bigr)
> 1-\epsilon_0,
\quad 
n = 1,2,\ldots,N.
\label{eq:truncation points 1 part}
\end{align}
We then have}
\begin{align}
&
(1-\epsilon_0)^n \Pr[\widehat{A}(t)=k] 
< 
\hat{p}_k^{\trunc}(t) 
\leq
\Pr[\widehat{A}(t)=k],
\nonumber
\\
&\hspace{10em}
k=0,1,\ldots,K,\ t \in (T_{n-1},T_n],\ n=1,2,\ldots,N.
\label{eq:lem:error}
\end{align}
\end{lemma}

\begin{proof}
\rrr{The second inequality in (\ref{eq:lem:error}) follows from 
(\ref{eq:trunc<original}) and (\ref{eq:hat{p}_k(t)}).
We thus 
consider the first inequality in \eqref{eq:lem:error} below.}

\red{
If we focus only on the number of arrivals, the uniformized Markov 
chain is reduced to a simple discrete-time birth process whose state 
transition diagram is given by Fig. \ref{fig:state transition diagram}.
By conditioning on both the cumulative number of arrivals up to time 
$T_{n-1}$ and the number of state transitions occurring during the 
interval $(T_{n-1},t]$, we obtain the following representation: 
\begin{align*}
\hat{p}_k(t)
&=
\Pr[\widehat{A}(t) = k]
\\
&=
\sum_{h=0}^{k}
\Pr[\widehat{A}(T_{n-1}) = h]
\Pr[\widehat{A}(t) - \widehat{A}(T_{n-1}) = k-h]
\\
&= 
\sum_{h=0}^{k} \hat{p}_h(T_{n-1})
\sum_{m=k-h}^{\infty} \Poi(\theta_n(t-T_{n-1}), m)\\
&\hspace{11em} \cdot
\binom{m}{k-h} 
\left(\frac{\lambda_n}{\theta_n}\right)^{k-h}
\left(1 - \frac{\lambda_n}{\theta_n}\right)^{m-k+h},
\\
&\hspace{8em}
k=0,1,\ldots,K,\ t \in (T_{n-1},T_n],\ n=1,2,\ldots,N.
\end{align*}
Moreover, from the correspondence between (\ref{eq:uniformization}) 
and (\ref{eq:computed value of transient state probability}), 
$\hat{p}_k^\trunc(t)$ is given by}
\begin{align}
\hat{p}_k^{\trunc}(t)
&= 
\sum_{h=0}^k \hat{p}_h^{\trunc}(T_{n-1})
\sum_{m=k-h}^{M_n} \Poi\bigl(\theta_n(t-T_{n-1}),m\bigr)
\nonumber
\\
&\hspace{12em}\cdot
\binom{m}{k-h} 
\left(\frac{\lambda_n}{\theta_n}\right)^{k-h}
\left(\frac{\theta_n-\lambda_n}{\theta_n}\right)^{m-k+h}
\nonumber
\\
&= 
\sum_{h=0}^k \hat{p}_h^{\trunc}(T_{n-1})
\sum_{m=k-h}^{M_n}
\Poi\bigl(\lambda_n (t-T_{n-1}),k-h\bigr)
\nonumber
\\
&\hspace{12em}
\cdot
\Poi\bigl((\theta_n-\lambda_n)(t-T_{n-1}),m-k+h \bigr)
\nonumber
\\
&= 
\sum_{h=0}^k \hat{p}_h^{\trunc}(T_{n-1})
\Poi\bigl(\lambda_n (t-T_{n-1}),k-h \bigr)
\nonumber
\\
&\hspace{6em}
\cdot
\sum_{m=0}^{M_n-k+h}
\Poi\bigl( (\theta_n-\lambda_n)(t-T_{n-1}),m\bigr)
\nonumber
\\
&\geq 
\sum_{h=0}^k \hat{p}_h^{\trunc}(T_{n-1})
\Poi\bigl(\lambda_n (t-T_{n-1}),k-h \bigr)
\nonumber
\\
&\hspace{6em}
\cdot
\sum_{m=0}^{M_n-K}
\Poi\bigl( (\theta_n-\lambda_n)(t-T_{n-1}),m\bigr),
\nonumber
\\
& \hspace{5em}
k=0,1,\ldots,K,\ t \in (T_{n-1},T_n],\ n=1,2,\ldots,N.
\label{eq:lem:error-proof}
\end{align}
\blue{
\rrr{Because} $(\theta_n - \lambda_n)(t - T_{n-1}) \le \theta_n (T_n -
T_{n-1})$ holds for any $t \in (T_{n-1},T_n]$, we have from
the stochastic ordering of Poisson distributions 
\cite[Theorem 1.A.13]{Shaked2007},
\begin{align}
\sum_{m=0}^{M_n-K}
\Poi((\theta_n - \lambda_n)(t - T_{n-1}),m)
&\geq
\sum_{m=0}^{M_n-K}
\Poi(\theta_n (T_n - T_{n-1}),m)
\nonumber
\\
&>
1 - \epsilon_0,
\label{eq:Poisson-ordering}
\end{align}
\rrr{where the second inequality follows from 
\eqref{eq:truncation points 1 part}.}}
Therefore, it follows from \eqref{eq:lem:error-proof} 
and \eqref{eq:Poisson-ordering} that
\begin{align}
\hat{p}_k^{\trunc}(t)
& >
(1-\epsilon_0) 
\sum_{h=0}^k \hat{p}_h^{\trunc}(T_{n-1})
\Poi\bigl(\lambda_n (t-T_{n-1}),k-h \bigr)
\nonumber
\\
&= 
(1-\epsilon_0) 
\sum_{h=0}^k \hat{p}_h^{\trunc}(T_{n-1})
\Pr[\widehat{A}(t)=k \mid \widehat{A}(T_{n-1})=h],
\nonumber
\\
& \hspace{5em}
k=0,1,\ldots,K,\ t \in (T_{n-1},T_n],\ n=1,2,\ldots,N.
\label{eq:lem:error-main}
\end{align}

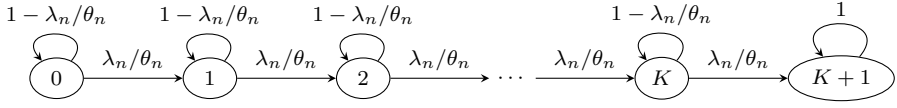
\begin{figure}[!t]
\centering
\begin{tikzpicture}%
[node/.style={draw, ellipse, minimum width=20pt, minimum height=15pt}]
\node[node](0){$0$};
\node[node, right = 1.3cm of 0](1){$1$};
\node[node, right = 1.3cm of 1](2){$2$};
\node[right = 1.3cm of 2](d){$\cdots$};
\node[node, right = 1.3cm of d](K){$K$};
\node[node, right = 1.3cm of K, font=\small](K+1){$K+1$};
\path[->, >=stealth]
(0) edge[above] node{$\lambda_n/\theta_n$} (1)
(1) edge[above] node{$\lambda_n/\theta_n$} (2)
(2) edge[above] node{$\lambda_n/\theta_n$} (d)
(d) edge[above] node{$\lambda_n/\theta_n$} (K)
(K) edge[above] node{$\lambda_n/\theta_n$} (K+1)
(0) edge[loop above, in=130, out=50, looseness=4] node{$1-\lambda_n/\theta_n$} (0)
(1) edge[loop above, in=130, out=50, looseness=4] node{$1-\lambda_n/\theta_n$} (1)
(2) edge[loop above, in=130, out=50, looseness=4] node{$1-\lambda_n/\theta_n$} (2)
(K) edge[loop above, in=130, out=50, looseness=4] node{$1-\lambda_n/\theta_n$} (K)
(K+1) edge[loop above, in=130, out=50, looseness=4] node{$1$} (K+1);
\end{tikzpicture}
\caption{The state transition diagram for the discrete-time Markov chain 
obtained by uniformization, focusing only on the number of arrivals, where  
$K+1$ represents the aggregated absorbing state with more than $K$ arrivals.}
\label{fig:state transition diagram}
\end{figure}

Using \eqref{eq:lem:error-main}, we prove the first inequality 
in (\ref{eq:lem:error}) by induction. 
Since (cf.\ (\ref{eq:trunc-initial}))
\[
\hat{p}_k^{\trunc}(T_0) = \begin{cases}
1, & k=0, \\
0, & \mbox{otherwise},
\end{cases}
\]
we have from (\ref{eq:lem:error-main}),
\[
\hat{p}_k^{\trunc}(t) > (1-\epsilon_0) \Pr[\widehat{A}(t)=k],
\quad
k=0,1,\ldots,K,\ t \in (0,T_1],
\]
so that (\ref{eq:lem:error}) holds for $n=1$. 
We now assume that 
\begin{align}
\hat{p}_k^{\trunc}(t) 
&> 
(1-\epsilon_0)^n \Pr[\widehat{A}(t)=k],
\quad
k=0,1,\ldots,K,\ t \in (T_{n-1},T_n],
\label{eq:induction-1}
\end{align}
holds for $n = m-1$ ($m \geq 2$). It then follows from
(\ref{eq:lem:error-main}) that
for $t \in (T_{m-1},T_m]$
\red{%
\begin{align*}
\hat{p}_k^{\trunc}(t)
& >
(1-\epsilon_0) 
\sum_{h=0}^k \hat{p}_h^{\trunc}(T_{m-1})
\Pr[\widehat{A}(t)=k \mid \widehat{A}(T_{m-1})=h]
\nonumber
\\
& >
(1-\epsilon_0) 
\sum_{h=0}^k (1-\epsilon_0)^{m-1} 
\Pr[\widehat{A}(T_{m-1})=h]
\\ &\hspace{15em} \cdot
\Pr[\widehat{A}(t)=k \mid \widehat{A}(T_{m-1})=h]
\nonumber
\\
& =
(1-\epsilon_0)^m
\Pr[\widehat{A}(t)=k],
\qquad
k=0,1,\ldots,K,
\end{align*}}
so that (\ref{eq:induction-1}) also holds for $n=m$, 
which completes the proof.
\qed
\end{proof}

\begin{theorem}\label{thm:error}
For a given $\epsilon$ $(\epsilon \in (0,1])$, 
if the truncation points $M_1, M_2, \ldots, M_N$ satisfy
\begin{equation}
\sum_{m=0}^{M_n-K}\Poi\bigl(\theta_n (T_n - T_{n-1}),m\bigr)
> 
(1-\epsilon)^{\frac{1}{N}},
\quad
n = 1,2,\ldots,N,
\label{eq:truncation points}  
\end{equation}
we have
\begin{equation}
\|\bm{\pi}(t,K)-\bm{\pi}^{\trunc}(t,K)\|_1 < \epsilon,
\label{eq:pi-error-bound}
\end{equation}
for all $t \in (0,T]$.
\end{theorem}

\begin{proof}
\red{
From \eqref{eq:truncation points} and Lemma \ref{lem:error 1 part},
we have for $t \in (T_{n-1}, T_n]$ ($n = 1,2,\ldots,N$), 
\begin{equation}
\hat{p}_k^\trunc(t) > (1-\epsilon)^{\frac{n}{N}}\Pr[\widehat{A}(t) = k].
\label{eq:hat-p-trunc-err}
\end{equation}
Therefore, using \eqref{eq:1-norm} and \eqref{eq:hat-p-trunc-err},
we obtain \eqref{eq:pi-error-bound} as follows:
\begin{align*}
\lefteqn{%
\| \bm{\pi}(t,K) - \bm{\pi}^{\trunc}(t,K)\|_1 
}\hspace*{6em}
\\
& <
1
-
(1-\epsilon)^{\frac{n}{N}}
\sum_{k=0}^K
\frac{\Poi\big(\Lambda(t,T),K-k\big)}
{\Poi\big(\Lambda(0,T),K\big)}
\Pr[\widehat{A}(t)=k]
\\
& \leq
1
-
(1-\epsilon)
\sum_{k=0}^K
\frac{\Poi\big(\Lambda(0,t),k\big) \Poi\big(\Lambda(t,T),K-k\big)}
{\Poi\big(\Lambda(0,T),K\big)}
\\
&=
1
-
(1-\epsilon)
\frac{\Poi\big(\Lambda(0,T),K\big)}
{\Poi\big(\Lambda(0,T),K\big)}
= \epsilon,
\end{align*}
where in the second inequality, we \rrr{use}
\[
(1-\epsilon)^{\frac{n}{N}} \geq 1-\epsilon,
\quad
n = 1,2,\ldots,N.
\]
}
\qed
\end{proof}

The computational procedure for $\bm{\pi}^{\trunc}(t,K)$ ($t \in
(0,T]$) is presented in Fig.\ \ref{fig:procedure}, 
where $N^*(t)$ denotes the natural number such that
\[
T_{N^*(t)-1} < t \leq T_{N^*(t)}.
\]

\begin{figure}
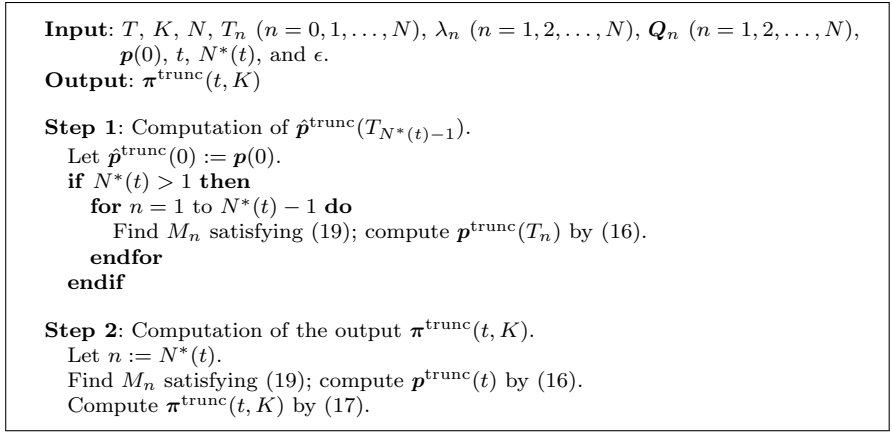

\centering
\noindent
\begin{small}
\fbox{\hspace*{3mm}
\begin{minipage}{110mm}
\mbox{}
\\
\textbf{Input}: 
$T$, $K$, $N$, $T_n$ ($n=0,1,\ldots,N$), 
$\lambda_n$ ($n=1,2,\ldots,N$), 
$\bm{Q}_n$ ($n=1,2,\ldots,N$), \\
\hspace*{3em}
$\bm{p}(0)$, $t$, $N^*(t)$, and 
$\epsilon$.
\\
\textbf{Output}: $\bm{\pi}^{\trunc}(t,K)$
\\[3mm]
\textbf{Step 1}: Computation of $\hat{\bm{p}}^{\trunc}(T_{N^*(t)-1})$.
\\
\q Let $\hat{\bm{p}}^{\trunc}(0) := \bm{p}(0)$.
\\ 
\q \textbf{if} $N^*(t) > 1$ \textbf{then}
\\
\q\q \textbf{for} $n=1$ to $N^*(t)-1$ \textbf{do}
\\
\q\q\q Find $M_n$ satisfying (\ref{eq:truncation points});\ 
compute $\bm{p}^{\trunc}(T_n)$ by 
(\ref{eq:computed value of transient state probability}).
\\
\q\q \textbf{endfor}
\\
\q \textbf{endif}
\\[3mm]
\textbf{Step 2}: Computation of the output $\bm{\pi}^{\trunc}(t,K)$.
\\
\q Let $n := N^*(t)$.
\\
\q Find $M_n$ satisfying (\ref{eq:truncation points});\ 
compute $\bm{p}^{\trunc}(t)$ by 
(\ref{eq:computed value of transient state probability}).
\\
\q Compute $\bm{\pi}^{\trunc}(t,K)$ by 
(\ref{eq:computed value of distribution of the number of customers}).
\\[-2mm]
\mbox{}
\end{minipage}}
\end{small}
\caption{Computational Procedure for 
$\bm{\pi}^{\trunc}(t,K)$ ($t \in (0,T]$).}
\label{fig:procedure}
\end{figure}

Although $A(T)=K$ and no customers arrive after time $T$, some
customers may remain in the system at time $t$, i.e., $\Pr[L(T) \geq 1] > 0$.
Suppose the auxiliary model after time $T$ is also formulated as a
continuous-time Markov chain $\{(\widehat{A}(t), \widehat{D}(t),
\widehat{S}(t))\}_{t > T}$ whose infinitesimal generator
$\widehat{\bm{Q}}_{N+1}$ takes the following form.
\[
\widehat{\bm{Q}}_{N+1} 
= 
\begin{bmatrix}
\bm{Q}_{N+1} & \bm{q}_{N+1}\\
\bm{0} & 0
\end{bmatrix}.
\]
Note here that all states with $\widehat{D}(t)=K$ are absorbing and
they are aggregated into a single absorbing state in
$\widehat{\bm{Q}}_{N+1}$.  We then have for $t > T$,
\begin{align*}
\hat{\bm{p}}(t) 
&= 
\hat{\bm{p}}(T) \exp[ \bm{Q}_{N+1} (t-T)]
\\
&=
\sum_{m=0}^{\infty}
\Poi(\theta_{N+1} (t-T),m)
\hat{\bm{p}}(T) \hat{\bm{P}}_{N+1}^m,
\end{align*}
where $\theta_{N+1}$ denotes the maximum of absolute values of
diagonal elements in $\bm{Q}_{N+1}$ and 
$\hat{\bm{P}}_{N+1} = \bm{I} +
\theta_{N+1}^{-1} \bm{Q}_{N+1}$.
To compute $\hat{\bm{p}}(t)$ for $t > T$, we truncate the infinite sum 
by $m=M_{N+1}$. 
\[
\hat{\bm{p}}^{\trunc}(t) 
=
\sum_{m=0}^{M_{N+1}}
\Poi(\theta_{N+1} (t-T),m)
\hat{\bm{p}}(T) \hat{\bm{P}}_{N+1}^m,
\quad
t > T.
\]
\blue{
For $t > T$, we define $\pi_\ell^\trunc(t,K)$ as
\[
\pi_\ell^\trunc(t,K)
=
\frac{\hat{p}_{K,K-\ell}^\trunc(t)}{\Poi(\Lambda(0,T),K)},
\quad
t > T,
\]
and the vectors $\bm{\pi}(t,K)$ and $\bm{\pi}^\trunc(t,K)$
similarly as before.}
It then follows that (cf. \eqref{eq:1-norm})
\[
\|\bm{\pi}(t,K) - \bm{\pi}^\trunc(t,K)\|_1
=
1
-
\frac{\hat{p}_K^\trunc(t)}{\Poi(\Lambda(0,T),K)},
\quad
t > T,
\]
where
\[
\hat{p}_K^\trunc(t)
=
\sum_{j=0}^{K}\hat{p}_{K,j}^\trunc(t),
\quad
t > T.
\]
The following corollary can be shown in a very similar way to 
Theorem \ref{thm:error}, so that we omit its proof.

\begin{corollary}
\label{cor:error-t>T}
For a given $\epsilon$ $(\epsilon \in (0,1])$ and $T_{\max}$
$(T_{\max}> T)$, if the truncation points $M_1, M_2, \ldots, M_N$,
and $M_{N+1}$ satisfy
\begin{align}
\sum_{m=0}^{M_n-K}\Poi\bigl(\theta_n (T_n - T_{n-1}),m\bigr)
&> 
(1-\epsilon)^{\frac{1}{N+1}},
\quad
n = 1,2,\ldots,N,
\label{eq:M_n}
\\
\sum_{m=0}^{M_{N+1}}\Poi\bigl(\theta_{N+1} (T_{\max}-T),m\bigr)
&> 
(1-\epsilon)^{\frac{1}{N+1}},
\label{eq:M_{N+1}}
\end{align}
we have
\[
\|\bm{\pi}(t,K)-\bm{\pi}^{\trunc}(t,K)\|_1 < \epsilon,
\]
for all $t \in (0,T_{\max}]$.
\end{corollary}

\red{
\begin{remark}
In \eqref{eq:M_{N+1}}, the upper limit of the summation is given by 
$M_{N+1}$ instead of $M_{N+1}-K$ as in \eqref{eq:M_n}. 
This difference is due to the following observation.
Because $\lambda(t) = 0$ for $t > T$, we have
\[
\hat{p}_K^\trunc(t)
=
\hat{p}_K^\trunc(T_N)
\sum_{m=0}^{M_{N+1}}
\Poi(\theta_{N+1}(t-T_N), m),
\quad
t > T.
\]
Therefore, if the truncation points $M_n$ 
($n = 1,2,\ldots, N, N+1$) satisfy \eqref{eq:M_n} and \eqref{eq:M_{N+1}}, 
then (cf.\ Lemma \ref{lem:error 1 part})
\[
(1-\epsilon)\Pr[\widehat{A}(t)=K]
<
\hat{p}_K^\trunc(t)
\leq
\Pr[\widehat{A}(t)=K],
\quad
t \in (T,T_{\max}],
\]
which leads to Corollary \ref{cor:error-t>T}.
\end{remark}
}

\section{Numerical examples for a Markovian queue}\label{sec:numerics}

In this section, we show some numerical examples of a CTBP/M/$c$
queue, where the CTBP is assumed to have a piecewise constant pdf
$f(t)$ and the service rate of each server is given by $\mu$. 
\red{%
In this example, $\widehat{\mathcal{S}}$ can be treated as a singleton
because the system state at any time $t$ is completely characterized by $(A(t), D(t))$.
}
In Appendix \ref{appendix:auxiliary}, we summarize the auxiliary 
model for this queue.

In all numerical examples, we set 
$T = 300$, $N = 30$, $T_n = 10n$ ($n=0,1,\ldots,30$),
$\mu = 2.5$, $c=2$, and $\epsilon = 10^{-14}$.
Moreover, we set
\begin{align}
f(t) = \Gamma n^2 e^{-0.25n},
\quad
t \in (T_{n-1},T_n],\ n=1,2,\ldots,30,
\label{eq:f(t)}
\end{align}
where $\Gamma$ denotes a normalizing constant, chosen such that 
\[
\Gamma = \bigg( \sum_{n=1}^{30} 10 n^2 e^{-0.25n} \bigg)^{-1}.
\]
In what follows, we consider $K=900$, 1,000, and 1,100, where $\alpha$
is fixed to 1,000. Figure \ref{fig:f(t)} shows $f(t)$ in
(\ref{eq:f(t)}) and the rate function $\lambda(t)=1000 f(t)$ of the
NHPP in the auxiliary model. In this setting, truncation points
$M_1,M_2,\ldots,M_{30}$ range from 1,212 to 1,301.

\begin{figure}[!t]
\centering
\subfigure[$f(t)$.]{%
\includegraphics[keepaspectratio,scale=0.7]{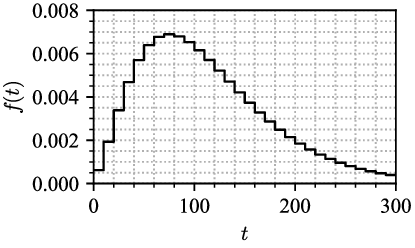}}
\subfigure[$\lambda(t)=1000 f(t)$.]{%
\includegraphics[keepaspectratio,scale=0.7]{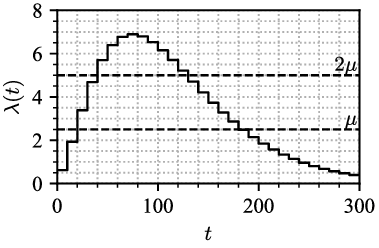}}
\caption{The pdf $f(t)$ and the rate functions $\lambda(t)$ in the NHPP.}
\label{fig:f(t)}
\end{figure}

For $K=900$, 1,000, and 1,100, Fig.\ \ref{fig:distribution} shows the
queue length distribution as a heatmap. 
\red{
The computation time required to generate each heatmap varies
significantly across hardware platforms: it took approximately 
5 hours on an Intel Core i7-9700K CPU 
with 16~GB RAM (Ubuntu~24.04), whereas it required only about 
3 minutes on a MacBook~Pro equipped with an Apple~M4 processor 
and 24~GB RAM.
}
We also plot 30 sample paths
of the queue length $L(t)$ obtained by simulation experiments. We
observe that most of the sample paths traverse the dark area of the
heat map, as expected.

\begin{figure}[!t]
\centering
\subfigure[$K=900$]{%
\includegraphics[bb= 0.000000 0.000000 194.400389 159.720319,%
keepaspectratio,scale=0.75]{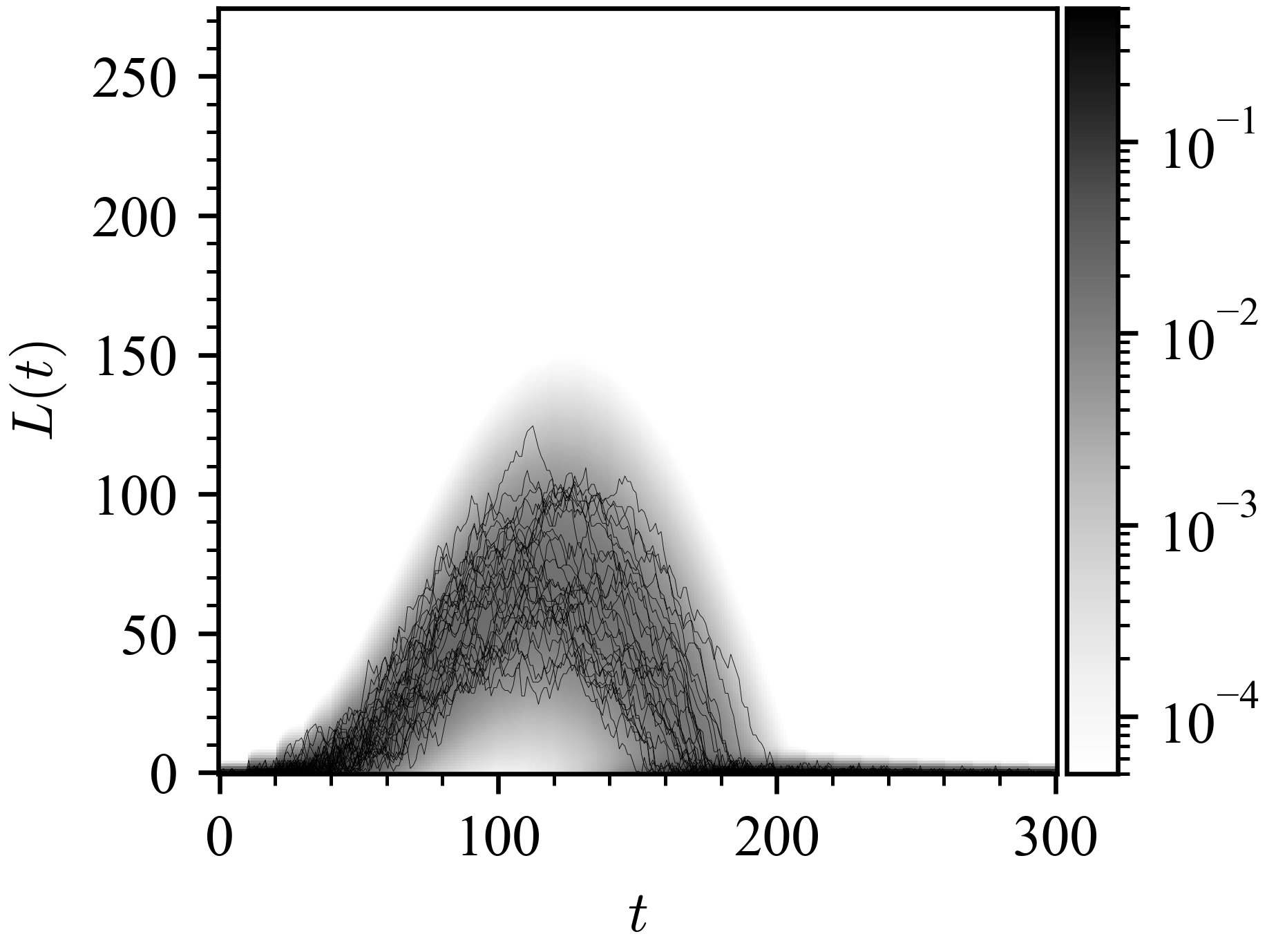}}
\subfigure[$K=1,000$]{%
\includegraphics[bb= 0.000000 0.000000 194.400389 159.720319,%
keepaspectratio,scale=0.75]{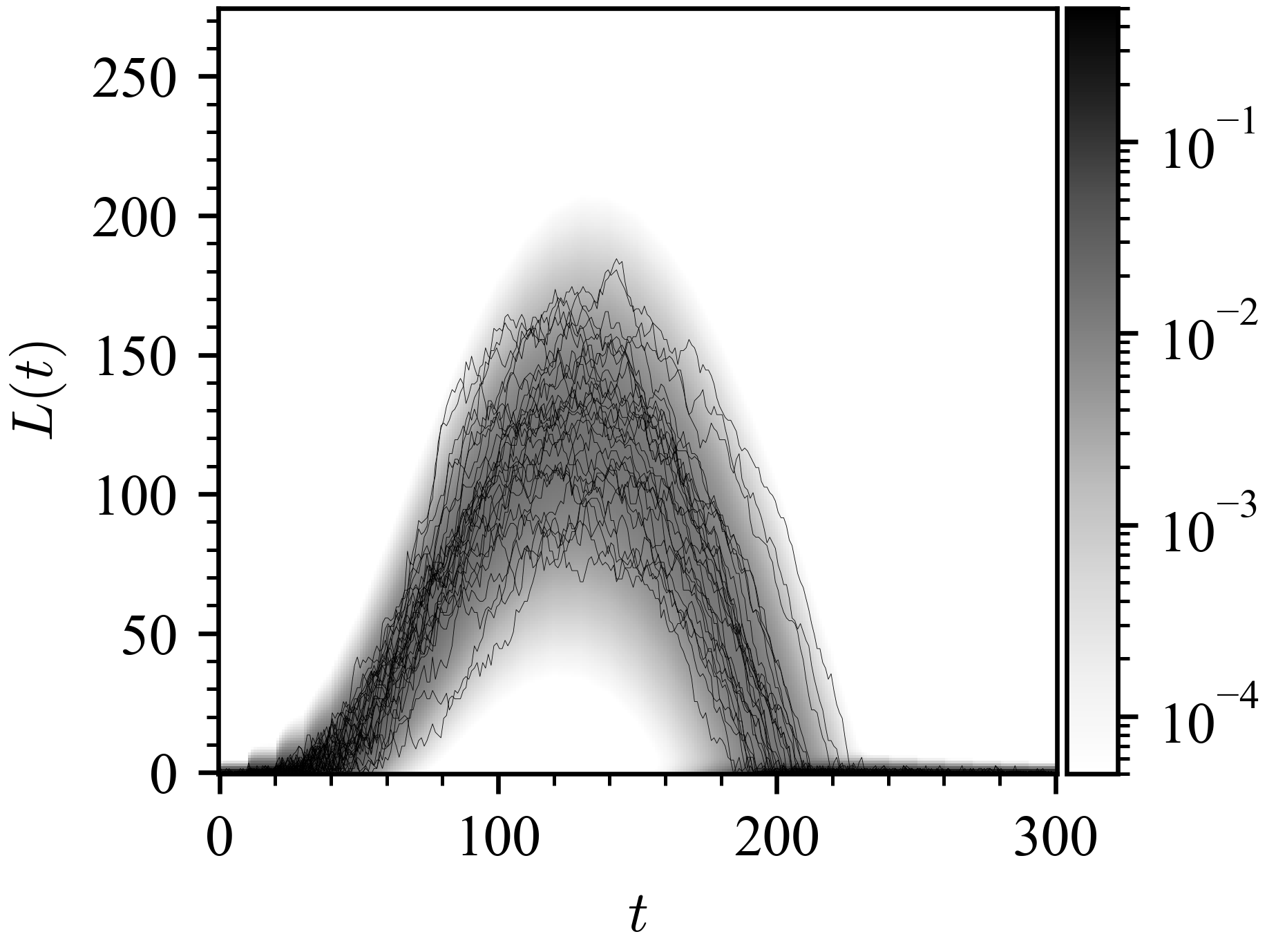}}
\newline
\subfigure[$K=1,100$]{%
\includegraphics[bb= 0.000000 0.000000 194.400389 159.720319,%
keepaspectratio,scale=0.75]{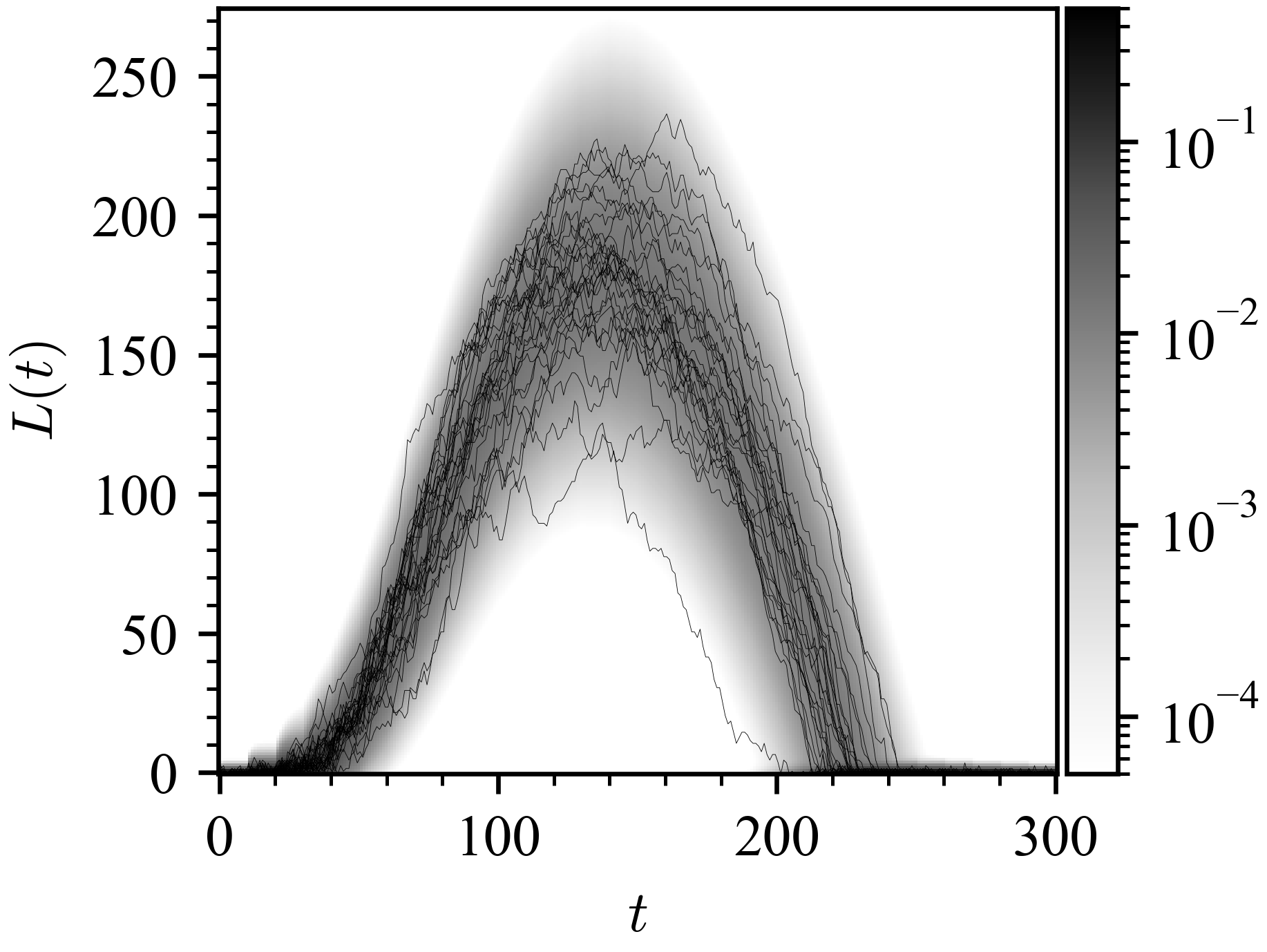}}
\caption{The distribution of the number of customers 
in the system and some sample paths.}
\label{fig:distribution}
\end{figure}

Figure \ref{fig:mean} shows the time-dependent mean queue length
$\E[L(t)]$. 
\red{We observe that $\E[L(t)]$ increases as $K$ increases, and the time at which 
$\E[L(t)]$ attains its maximum shifts to a later point for larger values of $K$.}
Note that a small change in the total number $K$ of arrivals has
a significant impact on the mean queue length. In particular, the
increase/decrease of 10\% in $K$ causes 40 to 50\% increase/decrease of
the maximum of the time-dependent mean queue length.

\begin{figure}[!t]
\centering
\includegraphics[keepaspectratio,scale=0.75]{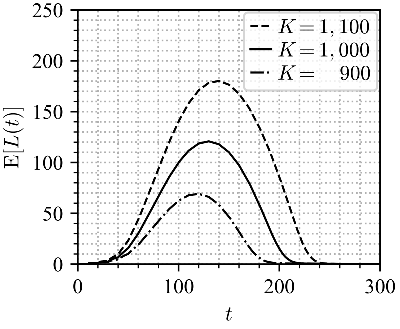}
\caption{The time-dependent mean queue length.}
\label{fig:mean}
\end{figure}

Lastly, we compare the CTBP/M/$2$ queue with the conventional
M${}_t$/M/$2$ queue with arrival rate function $\lambda(t) = Kf(t)$.
Figure \ref{fig:comparison of means with MtMc} shows the mean, median,
mode, and 95th percentile of the time-dependent queue length
distribution and Fig.\ \ref{fig:comparison of distribution with MtMc}
shows the queue length distributions at times $t = 50, 100, 150$, and
200. While there are no significant differences in the mean, median,
and mode in the two models, the time-dependent queue length in our
model is less variable than in the M${}_t$/M/$2$ queue.  In the M${}_t$/M/$2$
queue, the number $A(T)-A(t)$ of arrivals in $(t,T]$ is independent of
the number $A(t)$ of arrivals before time $t$, whereas in the CTBP,
$A(T)-A(t) = K- A(t)$, i.e., the numbers of arrivals before and after
time $t$ are negatively correlated. As a result, the time-dependent
queue length in the CTBP/M/$2$ queue is less variable than that of the
M${}_t$/M/$2$ queue.  Therefore, the model discussed in this paper
provides an essentially different perspective in performance
evaluation from that of the conventional queueing model.

\begin{figure}[!t]
\centering
\subfigure[Mean]{%
\includegraphics[keepaspectratio,scale=0.75]{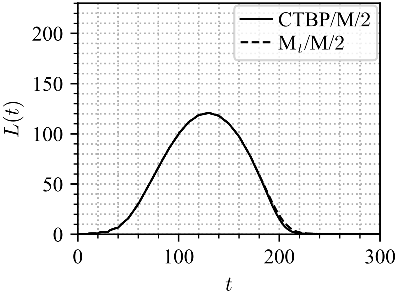}}
\subfigure[Median]{%
\includegraphics[keepaspectratio,scale=0.75]{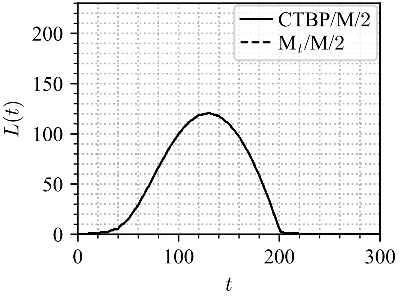}}
\\
\subfigure[Mode]{%
\includegraphics[keepaspectratio,scale=0.75]{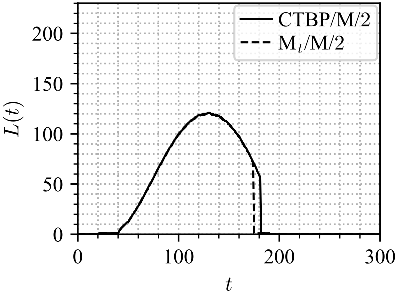}}
\subfigure[95th percentile]{%
\includegraphics[keepaspectratio,scale=0.75]{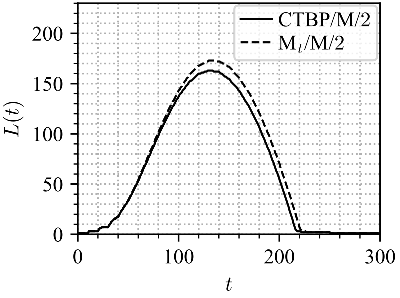}}
\caption{Comparison of the mean, median, mode, and 95th percentile 
of the time-dependent queue length distribution ($K=1,000$).}
\label{fig:comparison of means with MtMc}
\end{figure}
\begin{figure}[t]
\centering
\subfigure[$t=50$]{%
\includegraphics[keepaspectratio,scale=0.75]{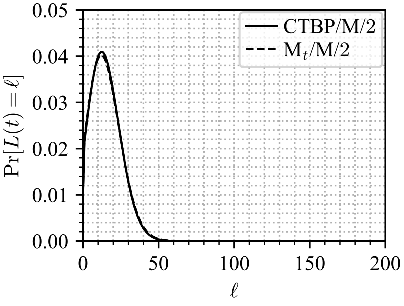}}
\subfigure[$t=100$]{%
\includegraphics[keepaspectratio,scale=0.75]{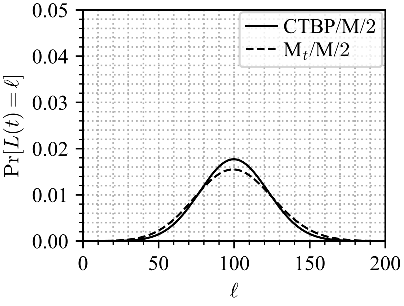}}
\\
\subfigure[$t=150$]{%
\includegraphics[keepaspectratio,scale=0.75]{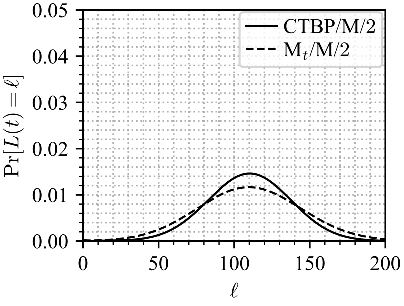}}
\subfigure[$t=200$]{%
\includegraphics[keepaspectratio,scale=0.75]{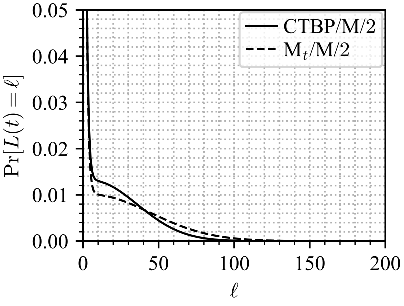}}
\caption{Comparison of the queue length distribution ($K=1,000$).}
\label{fig:comparison of distribution with MtMc}
\end{figure}

\section{Conclusion}\label{sec:conclusion}

In this paper, we considered queueing models with CTBP arrivals,
i.e., arrival times of $K$ customers are i.i.d.\ on a finite time
interval $[0, T]$. First, in a general framework, we showed that the
time-dependent queue length distribution in such a queueing model can
be expressed in terms of the time-dependent joint distribution of the
numbers of arrivals and departures in the auxiliary model with NHPP
arrivals.  Next, we presented a computational procedure for the
time-dependent queue length distribution in the \red{piecewise} Markovian case. A
notable feature of this procedure is that the truncation error bound
can be set as the input.  In numerical examples, we observed that the
total number $K$ of arrivals has a significant impact on the
time-dependent queue length distribution. Furthermore, the qualitative
difference between queues with CTBP arrivals and with NHPP arrivals
was clarified \red{(cf. Fig. \ref{fig:comparison of means with MtMc} 
and Fig. \ref{fig:comparison of distribution with MtMc})}.

We close this paper with some remarks on the computational aspects of
our results. Theorem \ref{thm:distribution of the number of customers}
shows that we have to compute the time-dependent joint pmf of the
numbers of arrivals and departures in the auxiliary model, and if
$f(t)$ ($t \in [0,T]$) is piecewise constant, the uniformization can
be utilized in principle, as shown in (\ref{eq:uniformization}).
Although we derived the upper bound of truncation error, the
computation of Poisson probabilities can be another source of
numerical error, especially for a large $K$.
\red{
To compute the pmf $\Poi(\Lambda(0,T),K)$ of the Poisson distribution
with sufficiently high accuracy, the parameter $\alpha$ must be 
chosen proportional to $K$ (cf.\ \eqref{eq:Lambda-F}).
On the other hand, when $\alpha$ is large (equivalently, when
$\Lambda(0,T)$ is large) the evaluation of
$\Poi(\Lambda(0,T),0) = e^{-\Lambda(0,T)}$ may suffer from exponent
underflow, making numerically stable computation nontrivial
\cite[pp. 19--20]{Tijms1994}.}
As regards this, one may
refer to \cite{Fox88}, where for a given error tolerance, the left and
right truncation points in the Poisson distribution can be
pre-determined. Besides, the computational cost can be problematic.
Because the dimension of the square matrix $\bm{P}_n$ is of order
$O(K^2 |\mathcal{S}|)$, the computational cost would be huge for a
large $K$ if we implement (\ref{eq:uniformization}) in a
straightforward manner. This problem can be mitigated by truncating
$\hat{\bm{p}}(T_{n-1})$ in (\ref{eq:uniformization}) based on the
observation that $\widehat{A}(T_{n-1})$ follows a Poisson distribution
with mean $\Lambda(0,T_{n-1})$. The development of a computational
procedure for large-scale systems remains as future work.

\red{
In this paper, we assumed that all customers arrive within the finite 
time interval $[0,T]$. In practical systems, however, it is often 
the case that some customers have already arrived before the service 
starts, and a model allowing arrivals prior to time $0$ has been 
considered in \cite{Boxma2025}. Our only assumption on 
the service mechanism is Assumption \ref{asm:service mechanism}, 
which does not require that service begins precisely at time $0$. 
Therefore, our framework can be extended to incorporate such models by 
appropriately modifying the initial-state probability vector $\bm{p}(0)$.}

\begin{acknowledgements}
This research was supported in part by JSPS KAKENHI Grant Number 24K14839.
\end{acknowledgements}

\appendix

\section{Proof of Lemma \ref{lem:distribution of cumulative arrivals}}
\label{appendix:lem:distribution of cumulative arrivals}

Without loss of generality, let $0 \leq t_1 \leq t_2 \leq \cdots \leq
t_m \leq t$.  In this setting, if $0 \leq k_1 \leq k_2 \leq \cdots
\leq k_m \leq k$ does not hold, both sides of
\eqref{eq:arrival_distribution} become $0$, so that
\eqref{eq:arrival_distribution} holds.  We thus assume $0 \leq k_1
\leq k_2 \leq \cdots \leq k_m \leq k \leq K$ below.  For convenience,
let $t_0=0$, $t_{m+1}=t$, $k_0=0$, $k_{m+1}=k$, $N(s,t) = A(t)-A(s)$,
and $\widehat{N}(s,t) = \widehat{A}(t)-\widehat{A}(s)$ ($0 \leq s \leq t \leq
T$). It then follows that
\begin{align}
\lefteqn{%
\Pr\bigl[A(t_1)=k_1, A(t_2)=k_2,\ldots,A(t_m)=k_m \mid A(t)=k\bigr] 
}\qquad
\notag
\\
&=
\frac{1}{\Pr[A(t)=k]}
\cdot
\Pr\bigl[
N(t_{i-1},t_i) = k_i-k_{i-1}
(i=1,2,\ldots,m+1),
N(t,T)=K-k
\bigr]
\notag
\\
&=
\bigg[ \binom{K}{k} \bigl(F(0,t)\bigr)^k\bigl(F(t,T)\bigr)^{K-k} \bigg]^{-1}
\notag
\\
&\qquad \cdot
\frac{\ds K!}
{\left(\ds\prod_{i=1}^{m+1}(k_i-k_{i-1})!\right)(K-k)!}
\cdot 
\bigg(
\prod_{i=1}^{m+1}\bigl(F(t_{i-1},t_i)\bigr)^{k_i-k_{i-1}}
\bigg)
\bigl( F(t,T) \bigr)^{K-k}
\notag\\
&=
\frac{k!}
{\ds \prod_{i=1}^{m+1}(k_i-k_{i-1})!}
\cdot
\frac{\ds\prod_{i=1}^{m+1}\bigl(F(t_{i-1},t_i)\bigr)^{k_i-k_{i-1}}}
{\bigl(F(0,t)\bigr)^k}.
\label{eq:arrivals in our model}
\end{align}

On the other hand, using the independent increment property of the
NHPP, we have
\begin{align}
\lefteqn{%
\Pr\bigl[
\widehat{A}(t_1) = k_1, \widehat{A}(t_2) = k_2,
\ldots,\widehat{A}(t_m) = k_m 
\mid \widehat{A}(t)=k
\bigr]
}\qquad
\notag
\\
&=
\frac{1}{\Pr[\widehat{A}(t)=k]}
\cdot
\Pr\bigl[
\widehat{N}(t_{i-1},t_i) = k_i-k_{i-1}
(i=1,2,\ldots,m+1)
\bigr]
\notag
\\
&=
\bigg[
e^{-\Lambda(0,t)} \frac{\bigl(\Lambda(0,t)\bigr)^k}{k!} 
\bigg]^{-1}
\prod_{i=1}^{m+1}
e^{-\Lambda(t_{i-1},t_i)}
\frac{\bigl(\Lambda(t_{i-1},t_i)\bigr)^{k_i-k_{i-1}}}
{(k_i-k_{i-1})!}
\notag
\\
&=
\frac{k!}{\ds \prod_{i=1}^{m+1}(k_i-k_{i-1})!}
\cdot
\frac{\ds\prod_{i=1}^{m+1}\bigl(\Lambda(t_{i-1},t_i)\bigr)^{k_i-k_{i-1}}}
{\bigl(\Lambda(0,t)\bigr)^k}.
\label{eq:arrivals in non-homogeneous Poisson}
\end{align}
Eq.\ (\ref{eq:arrival_distribution}) now follows from
\eqref{eq:Lambda-F}, \eqref{eq:arrivals in our model}, and
\eqref{eq:arrivals in non-homogeneous Poisson}.

When $t \geq T$, we have $\Pr[A(t)=K]=1$. Furthermore, 
\[
\widehat{A}(t) = K\ \  (t>T)\ 
\Leftrightarrow\ \widehat{A}(T) = K.
\]
These observations lead to (\ref{eq:arrival-condition:t=T}).

%%%%%%%%%%%%%%%%%%%%%%%%%%%%%%%

\section{The auxiliary model in numerical examples}
\label{appendix:auxiliary}

In the $n$-th interval $(T_{n-1}, T_n]$, the auxiliary model is
formulated as a time-homogeneous, continuous-time Markov chain
$\bigl(\widehat{A}(t),\widehat{D}(t)\bigr)$ whose transition rate
diagram is given by Fig.\ \ref{fig:transition rate}.  When the
states are arranged in lexicographical order, $\bm{Q}_n$ in
\eqref{eq:transition rate matrix} is given by
\[
\bm{Q}_n = 
\begin{bmatrix}
\bm{A}_{n,0} & \bm{B}_{n,0} & \bm{O}   & \cdots & \bm{O} & \bm{O} \\
\bm{O}   & \bm{A}_{n,1} & \bm{B}_{n,1} & \cdots & \bm{O} & \bm{O} \\
\bm{O}   & \bm{O}   & \bm{A}_{n,2} & \cdots & \bm{O} & \bm{O} \\
\vdots   & \vdots   & \vdots   & \ddots & \vdots & \vdots \\
\bm{O}   & \bm{O}   & \bm{O}   & \cdots & \bm{A}_{n,K-1} & \bm{B}_{n,K-1} \\
\bm{O}   & \bm{O}   & \bm{O}   & \cdots & \bm{O} & \bm{A}_{n,K} 
\end{bmatrix},
\]
where $\bm{A}_{n,k}$ is a $(k+1)\times(k+1)$ square matrix,
$\bm{B}_{n,k}$ is a $(k+1)\times(k+2)$ matrix, and their elements are
given by
\begin{align*}
(\bm{A}_{n,k})_{i,j} 
&= 
\begin{cases}
-\lambda_n - \min(k-i,c) \mu, & j=i,
\\
\min(k-i,c) \mu, & j=i+1,
\\
0, & \mbox{otherwise}.
\end{cases}
\\
(\bm{B}_{n,k})_{i,j} 
&= 
\begin{cases}
\lambda_n, & j=i,
\\
0, & \mbox{otherwise}.
\end{cases}
\end{align*}

\begin{figure}[!t]
\centering
\begin{tikzpicture}[node/.style={draw, ellipse, minimum width=30pt, minimum height=20pt}]
\node[node](0_0){$0,0$};
\node[node, right = 1.0cm of 0_0](1_0){$1,0$};
\node[node, below = 0.8cm of 1_0](1_1){$1,1$};
\node[node, right = 1.0cm of 1_0](2_0){$2,0$};
\node[node, below = 0.8cm of 2_0](2_1){$2,1$};
\node[node, below = 0.8cm of 2_1](2_2){$2,2$};
\node[right = 1.0cm of 2_0](dot0){$\cdots$};
\node[right = 1.0cm of 2_1](dot1){$\cdots$};
\node[right = 1.0cm of 2_2](dot2){$\cdots$};
\node[node, right = 1.0cm of dot0](K_0){$K,0$};
\node[node, below = 0.8cm of K_0](K_1){$K,1$};
\node[node, below = 0.8cm of K_1](K_2){$K,2$};
\node[below = 0.8cm of K_2](vdot){$\vdots$};
\node[below = 1cm of dot2](ddot){$\ddots$};
\node[node, below = 0.8cm of vdot](K_K){$K,K$};
\node[node, right = 1.0cm of K_2](K+1){$K+1$};

\path[->, >=stealth]
(0_0) edge[above] node{$\lambda_n$} (1_0)
(1_0) edge[above] node{$\lambda_n$} (2_0)
(1_1) edge[above] node{$\lambda_n$} (2_1)
(2_0) edge[above] node{$\lambda_n$} (dot0)
(2_1) edge[above] node{$\lambda_n$} (dot1)
(2_2) edge[above] node{$\lambda_n$} (dot2)
(dot0) edge[above] node{$\lambda_n$} (K_0)
(dot1) edge[above] node{$\lambda_n$} (K_1)
(dot2) edge[above] node{$\lambda_n$} (K_2)

(K_0) edge[above right, in=90, out=0] node{$\lambda_n$} (K+1)
(K_1) edge[above right, in=120, out=0] node{$\lambda_n$} (K+1)
(K_2) edge[above] node{$\lambda_n$} (K+1)
(K_K) edge[below right, in=270, out=0] node{$\lambda_n$} (K+1)

(1_0) edge[left] node{$\mu$} (1_1)
(2_0) edge[left] node{$2\mu$} (2_1)
(2_1) edge[left] node{$\mu$} (2_2)
(K_0) edge[left] node{$\min(K,c)\mu$} (K_1)
(K_1) edge[left] node{$\min(K-1,c)\mu$} (K_2)
(K_2) edge[left] node{$\min(K-2,c)\mu$} (vdot)
(vdot) edge[left] node{$\mu$} (K_K);
\end{tikzpicture}
\caption{State transition rate diagram of the auxiliary model,
where state $K+1$ is absorbing.}
\label{fig:transition rate}
\end{figure}
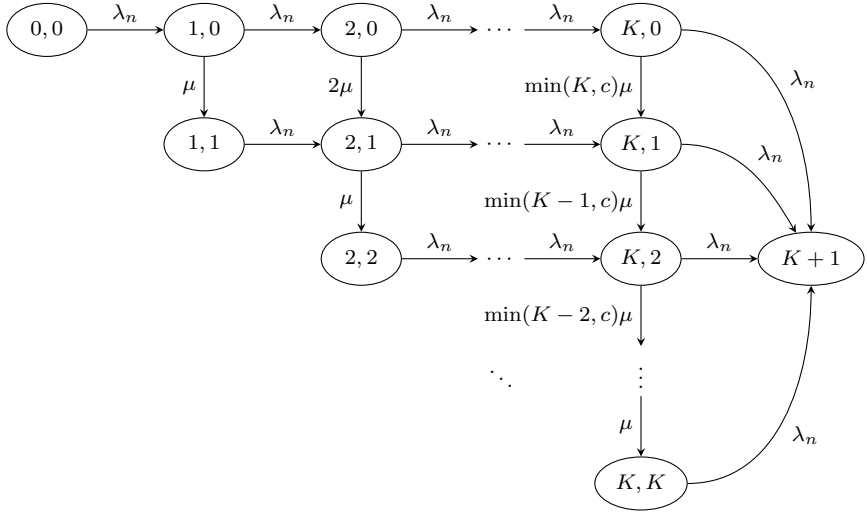

Because $\bm{Q}_n$ is sparse, the computation of 
$\hat{\bm{p}}^{\trunc}(T_{n-1}) \bm{P}_n^m$ ($m=0,1,\ldots$) in 
(\ref{eq:computed value of transient state probability})
can be simplified as follows. Let $\hat{\bm{p}}^{(m)}(T_{n-1})
= \hat{\bm{p}}^{\trunc}(T_{n-1}) \bm{P}_n^m$ ($m=0,1,\ldots$). We then have
\[
\hat{\bm{p}}^{(0)}(T_{n-1})
=
\hat{\bm{p}}(T_{n-1}),
\qquad
\hat{\bm{p}}^{(m)}(T_{n-1})
=
\hat{\bm{p}}^{(m-1)}(T_{n-1}) \bm{P}_n,
\quad
m=1,2,\ldots.
\]
Let $p^{(m)}_{k,j}(T_{n-1})$ ($k=0,1,\ldots,K$, $j=0,1,\ldots,k$) 
denote the $(k,j)$-th element of $\hat{\bm{p}}^{(m)}(T_{n-1})$.
It then follows that
\begin{align*}
p^{(m+1)}_{0,0}(T_{n-1}) 
&= 
\left(1-\frac{\lambda_n}{\theta_n}\right) 
p^{(m)}_{0,0}(T_{n-1}),
\quad
i = j = 0, 
\\
p^{(m+1)}_{i,0}(T_{n-1})
&= 
\frac{\lambda_n}{\theta_n} 
\cdot 
p^{(m)}_{i-1,0}(T_{n-1})
+ 
\left(
1-\frac{\lambda_n+\min(i,c)\mu}{\theta_n}
\right)
p^{(m)}_{i,0}(T_{n-1}),
\\
& \hspace*{12em}
1 \leq i \leq K,\ j = 0, 
\\
p^{(m+1)}_{i,j}(T_{n-1}) 
&= 
\frac{\lambda_n}{\theta_n} 
\cdot 
p^{(m)}_{i-1,j}(T_{n-1})
+ 
\frac{\min(i-j+1,c)\mu}{\theta_n} 
\cdot 
p^{(m)}_{i,j-1}(T_{n-1}) 
\\
& \qquad {} 
+ 
\left(
1-\frac{\lambda_n+\min(i-j,c)\mu}{\theta_n}
\right)
p^{(m)}_{i,j}(T_{n-1}),
\\
& \hspace*{10em}
1 \leq i \leq K,\ 1 \leq j \leq i-1,
\\
p^{(m+1)}_{i,i}(T_{n-1})
&= 
\frac{\mu}{\theta_n} 
\cdot 
p^{(m)}_{i,i-1}(T_{n-1})
+ 
\left(
1-\frac{\lambda_n}{\theta_n}
\right) 
p^{(m)}_{i,i}(T_{n-1}),
\\
& \hspace*{12em}
1 \leq i \leq K,\ j = i,
\end{align*}
where $\theta_n = \lambda_n + c \mu$.

%%%%%%%%%%%%%%%%%%%%%%%%%%%%%%%%%%%%%%%%%%%%%%%%%%

%\begin{acknowledgements}
%If you'd like to thank anyone, place your comments here
%and remove the percent signs.
%\end{acknowledgements}

% BibTeX users please use one of
%\bibliographystyle{spbasic}      % basic style, author-year citations
%\bibliographystyle{spmpsci}      % mathematics and physical sciences
%\bibliographystyle{spphys}       % APS-like style for physics
%\bibliography{}   % name your BibTeX data base

% Non-BibTeX users please use
% \begin{thebibliography}{}
% %
% % and use \bibitem to create references. Consult the Instructions
% % for authors for reference list style.
% %
% \bibitem{RefJ}
% % Format for Journal Reference
% Author, Article title, Journal, Volume, page numbers (year)
% % Format for books
% \bibitem{RefB}
% Author, Book title, page numbers. Publisher, place (year)
% % etc
% \end{thebibliography}

\end{document}